\documentclass[a4paper,11pt, reqno]{amsart}

\usepackage{geometry}
\usepackage{hyperref}
\usepackage{amsmath,amsfonts}
\usepackage{amsthm}
\usepackage{amssymb}
\usepackage{multirow}
\usepackage{float}
\usepackage{adjustbox}
\usepackage{subcaption}
\def\R{\mathbb{R}}
\def\d{\mathrm{D}}
\def\B{\mathbb{B}}
\def\O{\mathbb{O}}
\def\Z{\mathbb{Z}}
\def\A{\mathcal{A}}
\def\S{\mathcal{S}}
\def\M{\mathcal{M}}
\def\L{\mathcal{L}}
\def\C{\mathbf{C}}
\def\K{\mathbf{K}}
\newcommand{\dsum}{\displaystyle\sum}

\newtheorem{theorem}{Theorem}
\newtheorem{cor}[theorem]{Corollary}
\newtheorem{lem}[theorem]{Lemma}

\newtheorem{example}[theorem]{Example}

\usepackage[ruled,vlined]{algorithm2e}
\usepackage{tikz}

\usepackage{pgfplots}
\pgfplotsset{compat=newest}
\usetikzlibrary{decorations.markings}
\let\origmaketitle\maketitle
\def\maketitle{
  \begingroup
  \def\uppercasenonmath##1{} 
  \let\MakeUppercase\relax 
  \origmaketitle
  \endgroup
}

\title{Continuous maximal covering location problems\\with interconnected facilities}

\author[V\MakeLowercase{ictor} B\MakeLowercase{lanco} \MakeLowercase{and} R\MakeLowercase{icardo} G\MakeLowercase{\'azquez}]{{\large V\'ictor Blanco and Ricardo G\'azquez}\medskip\\
IEMath-GR, Universidad de Granada\\
\texttt{vblanco@ugr.es}, \texttt{rgazquez@ugr.es}
}

\begin{document}

\maketitle

\begin{abstract}
In this paper we analyze a continuous version of the maximal covering location problem, in which the facilities are required to be interconnected by means of a graph structure in which two facilities are allowed to be linked if a given distance is not exceed. We provide a mathematical programming framework for the problem and different resolution strategies. First, we propose a Mixed Integer Non Linear Programming formulation, and derive properties of the problem that allow us to project the continuous variables out avoiding the nonlinear constraints, resulting in an equivalent pure integer programming formulation. Since the number of constraints in the integer programming formulation is large and the constraints are, in general, difficult to handle, we propose two branch-\&-cut approaches that avoid the complete enumeration of the constraints resulting in more efficient procedures. We report the results of an extensive battery of computational experiments comparing the performance of the different approaches.
\end{abstract}
\section{Introduction}
The goal of Covering Location problems is to locate a set of facilities to give service to a finite set of users provided that the facilities are allowed to satisfy the demand of the users within certain coverage distance. The most popular covering location problems are the $p$-center \cite{hakimi64} and the maximal covering location~\cite{church1974maximal} problems. In the $p$-center problem, all users are assumed to be covered, by minimizing the maximum distance between the customers and their closest facilities, being then useful for locating emergency services in which every user should be located ``close'' to an available facility to be attended. In this case, the coverage distances for each facility are determined by the distance to its furthest user. In contrast, in the maximal covering location problem the coverage areas are given and the goal is to determine the optimal position of the facilities to cover the maximal (weighted) number of customers. This problem has applications in the location of warehouses, health-care centers, fire stations, etc. ~\cite{church1974maximal}, but also in other types of problem not directly related with logistics, as for instance, in data abstraction, selection of stocks, classification or grouping problems (see ~\cite{chung86}). The interested readers are referred to the recent book by Church and Murray~\cite{church2018location} or the chapter \cite{sergio-alfredo2020} for further details on covering location problems. 

Covering Location problems can be analyzed within the two main different frameworks in Location Analysis: discrete and continuous. In the discrete case, a given finite set of potential facilities is provided, and one has to choose, among them, the optimal facilities. These problems are useful in case the facilities to be located are physical services as ATMs, stores, hospitals, etc., in which the potential locations can be previously determined by the decision maker. However, in telecommunication networks, the positions of routers, alarms, etc. are more flexible to be positioned, and a continuous framework in which the facilities are allowed to be locate in the whole space (or in restricted regions of it), is known to be more appropriate. The continuous framework is also useful to determine the set of potential facilities that serve as input for a discrete version of the problem. The main difference between the two families of problems is that in the discrete case the costs between the facilities and the users are given as input data (or can be preprocessed before solving the problem), but in the continuous case, the costs/distances are part of the decision and have to be incorporated into the optimization problem. While in the discrete case, Integer Linear Programming tools play an important role when trying to solve the problems, the continuous problems are usually solved using Global Optimization tools.

On the other hand, in many situations when locating facilities, it is desirable that they are interconnected, that is, linked if the distance between them does not exceed a given limit. It is the case of the optimal design of forest fire-fighters centers that have to communicate a central server at a give radius \cite{demaine09} or the location of sensors that have to be connected to each others \cite{romich15}. Some facility location problems have been analyzed under these interconnection lens (see \cite{BPE16} in which continuous multiple-allocation multifacility ordered median location problems are analyzed and the distances between the location of the facilities are considered in the objective function), although interconnection between services have been mostly studied in the context of hub location, in which routing costs induce such a connectivity between the hub nodes (see \cite{HubChapter19} and the references therein). The $p$-median and the $p$-maximal discrete covering location problems with tree-shaped interconnected facilities have been recently studied in \cite{Landete19}. However, the continuous version of the problem, as far as we know, has not been previously studied in the literature. As already mentioned, continuous location problems are useful in case the facilities can be arbitrary located at any position in the given space, and then, mostly applicable in telecommunication networks. It is usual in when deciding the situation of telecommunication servers in order to provide a more accurate service, by sharing capacity with other servers or providing service to customers allocated to another facility in case of failure of the one to which is allocated. 

In this paper we analyze a novel version of the continuous maximal covering location problem, in which we require the facilities to be interconnected (i.e., not exceeding the interconnecting distance) by a given graph structure. This graph has implications in the way the facilities are linked. For instance, assuming that the graph of facilities is complete, implies that all the facilities must be linked between them, and then, in case of failure of a facility, all the others can act as a backup for it. In case the facilities are linked with matching-shape, each facility is connected to a single other facility. In cycle-shaped interconnections, the flow between facilities runs in circle. We explore also other cases as star or ring-star shapes that are useful when one facility serves as a main server linked to all the others.

We propose a general framework for the Maximal Covering Location Problem with Interconnected Facilities (MCLPIF) in a $d$-dimensional space, in which (users-to-facility-users and facility-to-facility) distance/costs are measured by means of $\ell_\tau$ ($\tau\geq 1$) or polyhedral norms, We propose a Mixed Integer Non Linear Programming (MINLP) formulation for the problem and show how it can be reformulated as a Mixed Integer Second Order Cone Optimization (MISOCO) problem that can be solved using any if the available MISOCO solvers (Gurobi, CPLEX, XPRESS, etc.). Then, we provide a decomposition of the problem in which the continuous variables and the nonlinear constraints are projected out, and where the difficulty of the problem turns in solving a pure Integer Linear Programming (ILP) problem by exploiting the geometry of the problem. Although the ILP formulation is a valid formulation for the problem, its implies checking a large number of intersections of sophisticated convex regions in a $d$-dimensional space. Thus, we concentrate on analyzing the problem on the plane and using Euclidean distances to derive a suitable ILP formulation for the problem and two branch and cut approaches for solving it efficiently. 

The rest of the paper is organized as follows. In Section \ref{sec:2} we recall the classical Maximal Covering Location Problem and introduce the Maximal Covering Location Problem with Interconnected Facilities (MCLPIF). There, we provide a MINLP formulation for the MCLPIF and a MISOCO reformulation for it. Section \ref{sec:3} is devoted to the reformulation of the problem as an ILP problem by exploiting the geometry of the problem on the plane and Euclidean norms. In Section \ref{sec:4} we derive two branch-and-cut approaches for solving the MCLPIF by means of relaxing subsets of constraints of the ILP formulation. The results of our computational experience are reported in Section \ref{sec:5} showing the performance of the different approaches proposed for the problem. Finally, we draw some conclusions and point out some further research topics in Section \ref{sec:6}.

\section{Maximal Covering Location Problem with Interconnected Facilities}\label{sec:2}

In this section we introduce the problem under analysis and fix the notation for the rest of the sections. First, we recall the classical version of the continuous maximal covering location problem whose planar version was first introduced by Church~\cite{church1984}.

Let $\A = \{a_1, \ldots, a_n\} \subseteq \R^d$ be a set of demand points in a $d$-dimensional space, a set of weights $\omega_1, \ldots, \omega_n >0 $, $p \in \Z_+$, $R_i\geq 0$ for $i =1, \ldots, n$ and $p \in \Z_+$ with $p\geq 1$. Let $\d: \R^d \times \R^d \rightarrow \R_+$ be a distance measure in $\R^d$. The goal of the (Continuous) Maximal Covering Location Problem (MCLP) introduced by Church and Revelle~\cite{church1974maximal}, is to determine the positions of $p$ new facilities (centers) in $\R^d$, $X_1, \ldots, X_p$ that maximizes the ($\omega$-weighted) number of points covered by the facilities, provided that the demand point $a_i$ is covered if a facility is at most at distance $R_i$ to it, for $i=1, \ldots, n$.

We denote by $N=\{1, \ldots, n\}$ and $P=\{1, \ldots, p\}$ the index sets for the demand points and the centers, respectively. A mathematical programming formulation for the problem can be derived by using the following set of binary variables:
$$
z_{ij} = \left\{\begin{array}{cl} 
1 & \mbox{if $i$ is covered by the $j$-th facility,}\\
0 & \mbox{otherwise}
\end{array}\right. \quad  \text{for all $i\in N$, $j\in P$}.
$$
Concretely, the MCLP can be formulated as:
          \begin{subequations}
    \makeatletter
        \def\@currentlabel{${\rm MCLP}$}
        \makeatother
       \label{MCLP}
        \renewcommand{\theequation}{${\rm MCLP}_{\arabic{equation}}$}
\begin{align}
\max &\dsum_{i \in N} \omega_i \dsum_{j \in P} z_{ij} \nonumber\\
\mbox{s.t. } & \d(X_j, a_i) \leq R_i, \mbox{ if $z_{ij}=1$}, \forall i \in N, j \in P, \label{pmclp:1}\\
& \dsum_{j \in P} z_{ij} \leq 1, \forall i \in N,\label{pmclp:2}\\
& z_{ij} \in \{0,1\}, \forall i\in N, j \in P,\nonumber\\
& X_1, \ldots, X_p \in \R^d.\nonumber
\end{align}
\end{subequations}
where the objective function accounts for the weighted number of covered points, constraints \eqref{pmclp:1} ensure that covered points are those with a center in their coverage radius and constraints \eqref{pmclp:2} enforce that covered demand points are accounted only once in the objective function, even if it can be covered by more than one center. Note that constraint \eqref{pmclp:1} can be equivalently rewritten as:
$$
\d(X_j, a_i) \leq R_i + M(1- z_{ij}), \forall i \in N, j \in P.
$$
for a big enough constant $M> \max_{i,j \in N} \d(a_i, a_j)$.

The above formulation is clearly discrete and non linear, but the nonlinear constraints can be efficiently reformulated as a set of second order cone constraints~(see \cite{BEP14}), resulting in a Mixed Integer Second Order Cone Optimization (MISOCO) problem. However, for the Euclidean planar MCLP, Church\cite{church1984} proved that it is enough to inspect a explicit finite set of potential centers to find the optimal location of the problem. In particular, the so-called Circle Intersection Points (CIP) which consists of the demand points and the pairwise intersection of the balls (disks) centered at the demand points and the corresponding covering radia. Thus, the problem turns into the discrete version of the Maximal Covering Location Problem:
\begin{align}
\max &\dsum_{i \in N} \omega_i y_i\nonumber\\
\mbox{s.t. } & \dsum_{j \in \{1, \ldots, K\}:\atop \d(c_j, a_i) \leq R_i} x_j \geq y_i, \forall i \in N, \nonumber\\
& \dsum_{j \in \{1, \ldots, K\}} x_j = p,\nonumber\\
& y_i \in \{0,1\}, \forall i \in N,\nonumber\\
& x_j \in \{0,1\}, \forall j=1, \ldots, K.\nonumber
\end{align}
where ${\rm CIP}=\{c_1, \ldots, c_k\}$ is the set of Circle Intersection Points, $y_i$ is a binary variable that takes value one if point $a_i$ is covered by a chosen center, and zero otherwise, and $x_j$ is a binary variable taking value one if the point $c_j$ in ${\rm CIP}$ is chosen as a center. 

The above model, although allows a linear representation of the planar MCLP may has a large number of $x$-variables, in case the CIP set is large (in worst case, $O(n^2)$). Another equivalent linear representation of the problem, in case we assume that centers must cover at least one demand point, and that we exploit in this paper, is based on the following straightforward observation where we denote by $\B_R(a) = \{X \in \R^d: \d(a,X)\leq R\}$ the $\d$-ball of radius $R$ centered at $a \in \R^d$.
\begin{lem}\label{lem:1}
Let $C_1, \ldots, C_p \subseteq N$ be $p$ nonempty disjoint subsets of $N$. Then, they induce a solution to the MCLP (where $C_j$ are the points covered by the $j$-th center) if and only if:
$$
\bigcap_{i \in C_j} \B_{R_i}(a_i) \neq \emptyset, \forall j \in P.
$$
\end{lem}
The above result allows us to rewrite constraints \eqref{pmclp:1} as linear constraints and formulate the MCLP as:
\begin{align}
\max &\dsum_{i \in N} \omega_i \dsum_{j \in P} z_{ij} \nonumber\\
\mbox{s.t. } & \dsum_{i \in S} z_{ij} \leq |S|-1, \forall j \in P \text{ and } S \subseteq N: \bigcap_{i \in S} \B_{R_i}(a_i) = \emptyset, \label{pmclp:5}\\
& \dsum_{j \in P} z_{ij} \leq 1, \forall i \in N,\nonumber\\
& \dsum_{i \in N} z_{ij} \geq 1, \forall j \in P,\nonumber\\
& z_{ij} \in \{0,1\}, \forall i\in N, j \in P,\nonumber
\end{align}
where \eqref{pmclp:5} enforces that the set of points covered by a center must verify the condition of Lemma \ref{lem:1}. Once the solution of the problem above is obtained, $z^*$, explicit coordinates of the centers, can be found in the following sets:
$$
X_j \in \bigcap_{i \in N:\atop z^*_{ij}=1} \B_{R_i}(a_i), \forall j \in P,
$$
which can be formulated as a convex feasibility problem:
$$
\begin{aligned}[t]
\min \;\; & 0\\
\mbox{s.t. } &  \d(X_j,a_i) \leq R_i, &\forall i \in N: z^*_{ij}=1, \label{pmclp:6}\\
& X_j \in \R^d.&
\end{aligned}
$$
for all $j\in P$. 

The convex problem above can be efficiently handled for the most common distance measures, in particular for the Euclidean norm and the $\ell_1$ or $\ell_\infty$-norms, but also for any $\ell_\tau$-norm with $\tau$ (see \cite{BEP14} for the specific reformulation of constraints \eqref{pmclp:6} as Second Order Cone constraints). Thus, for this wide family of distance measures, and once the discrete part of the problem is solved, the coordinates of the centers can be solved in polynomial-time by using interior-point techniques~\cite{NN94}.

For the sake of simplicity, we assume in this paper that $\d$ is the Euclidean distance, although the analysis can be adapted for any other $\ell_\tau$-distance with $\tau \geq 1$.

\vspace*{0.5cm}

{\bf The MCLP with Interconnected Facilities}

\vspace*{0.5cm}

Apart from the input data provided for the MCLP, in the Maximal Covering Location Problem with Interconnected Facilities (MCLPIF), links between facilities are to be constructed following the specification of a graph structure. This extra requirement is useful in many situations in which communication between facilities is desirable, as for instance, in the design of telecommunication networks in which the servers to be located have a restricted coverage area to cover the users, but at the same time, they have to communicate other facilities with certain coverage radius.

We denote by $G$ the undirected complete graph with nodes $P=\{1, \ldots, p\}$ and $\S(G)$ certain undirected spanning subgraph of $G$. This subgraph structure is flexible enough and must include the particularities of the network structure for the facilities that the decision maker desires to impose to the solution of the problem. 

With the notation above, the goals of (MCLPIF) are:
\begin{itemize}
\item To determine the positions of $p$ facilities, $X_1, \ldots, X_p \in \R^d$ that maximize the weighted covered points such that the coverage area of the facilities is determined by the radia $R_i$, for $i\in N$, and
\item To decide the activated links between the new facilities following the particularities of the structure $\S(G)$, provided that two facilities can be interconnected if its distance, $\d$, between them does not exceed the radius $r\geq 0$.
\end{itemize}

\begin{lem}
The MCLPIF is NP-hard.
\end{lem}
\begin{proof}
The MCLPIF reduces to the MCLP in case $\S(G)$ is a $0$-connected subgraph or $\S(G)$ is any spanning subgraph and $r$ big enough, i.e., any link is possible between the facilities. Thus, by \cite{church1984}, the problem reduces to a discrete MCLP which is known to be NP-hard~\cite{megiddo83}.
\end{proof}
The structure of the spanning subgraph that wants to be constructed is assumed to be given by the decision maker. In \cite{Landete19} the authors consider a tree structure for the facilities since it is the least costly connected spanning structure on $G$ and the cost is not considered in the formulation. Here we provide a general framework for different graph structures, although we will concentrate in those that can be directly imposed to the model, as we will show later.
                                                                                                           
To provide a suitable mathematical programming formulation for the MCLPIF, apart from the decisions on the demand points that are covered by the facilities, i.e., the $z$-variables used in the previous models, we use the following set of binary variables concerning the activated links between facilities:                                                                                           
$$
x_{jk} = \left\{\begin{array}{cl} 
1 & \mbox{if  $k$ and $j$ are linked in the facilities graph $\S(G)$,}\\
0 & \mbox{otherwise}
\end{array}\right. \mbox{ for all $j, k \in P$}.
$$

Thus, with this set of variables, the MCLPIF can be formulated as the following Mixed Integer Non Linear Programming problem, that we call \eqref{NL}:
          \begin{subequations}
    \makeatletter
        \def\@currentlabel{${\rm NL}$}
        \makeatother
       \label{NL}
        \renewcommand{\theequation}{${\rm MINL}_{\arabic{equation}}$}
\begin{align}
	\max & \dsum_{i\in N}  \omega_i \dsum_{j \in P} z_{ij} \nonumber\\
	\mbox{s.t. } & \dsum_{j \in P} z_{ij} \leq 1, \forall i \in N, \nonumber\\
 			 & \d(X_j, a_i) \leq R_i, \mbox{ if $z_{ij}=1$}, \forall i \in N, j \in P, \nonumber\\
			 & \d(X_j, X_k) \leq r, \mbox{ if $x_{jk}=1$}, \forall j, k \in P, \label{pmclp:8}\\
			& x \in \mathcal{S}(G),\label{c1:4}\\
			& z_{ij}\in \{0,1\}, \forall i \in N, j \in P,\nonumber\\
			& x_{jk}\in \{0,1\}, \forall j, k \in P, j<k, \nonumber\\
			& X_j \in \R^2, \forall j \in P.\nonumber
\end{align}
    \end{subequations}
In this case, new decisions are incorporated to those of the MCLP. In particular, constraint \eqref{pmclp:8} ensures that facilities are allowed to be linked only in case the distance between them is smaller or equal to $r$; and, abusing of notation, in \eqref{c1:4} we incorporate all the constraints ensuring the desired properties of the spanning subgraph of facilities $\S(G)$. As in the case of \eqref{pmclp:1}, constraint \eqref{pmclp:8} can be equivalently rewritten as:
$$
\d(X_j, X_k) \leq r + M(1- x_{jk}), \forall j, k \in P.
$$
for a big enough constant $M$.

\subsection{Spanning Subgraph of Facilities}\label{subgraphs}

The subgraph $\S(G)$ is assumed to reflect the requirements imposed by the decision maker on the facilities network. Several options are possible when designing such a network. Since the goal of this paper is to provide a general framework for the maximal covering location problem with interconnected facilities, we will concentrate on network structures that can be imposed to the model, through constraint \eqref{c1:4}, by directly fixing the values of the $x$-variables. However, at the end of this section we will describe other possibilities that can be considered.

In what follows we describe some graph structures that can be incorporated to the model just by fixing the values of the $x$-variables. Observe that in a continuous location problem, as the MCLPIF, the \textit{labels} $\{1, \ldots, p\}$ given to the facilities are arbitrary and just allows to group the users, but any permutation of the labels is also a solution of the problem. Taking into account this consideration we consider the following six spanning graph structures:

\begin{itemize}
\item {\bf Complete}. The case of a complete graph is easy to model just by fixing the values of all the $x$-values to one. Since facilities must simultaneously cover at least one point and be linked with the rest of facilities, the problem may be infeasible. 

\item {\bf Cycle}. For cycles, since the indices of the facilities have no implications on the demand points (once a solution is computed, a permutation of the indices in the facilities is also an optimal solution), the cycle property can be imposed as:
$$
x_{jk} = \left\{\begin{array}{cl} 1 & \mbox{if ($k=j+1$) or ($j=1$ and $k=p$)},\\
0 & \mbox{otherwise}.
\end{array}\right.
$$
\item {\bf Line}. Similarly to cycles, one can impose a line just by setting:
$$
x_{jk} = \left\{\begin{array}{cl} 1 & \mbox{if $k=j+1$},\\
0 & \mbox{otherwise}.
\end{array}\right.
$$
\item {\bf Star}. The star shape on the sets of nodes $\{1, \ldots, p\}$ can be enforced by fixing the central node of the star to $1$ and the links as:
$$
x_{jk} = \left\{\begin{array}{cl} 1 & \mbox{if $j=1$ and $k\neq 1$},\\
0 & \mbox{otherwise}.
\end{array}\right.
$$
\item {\bf Ring-Star}. The structure of a ring star graph on $\{1, \ldots, p\}$ can incorporated, similarly to the star as:
$$
x_{jk} = \left\{\begin{array}{cl} 1 & \mbox{if ($j=1$ and $k\neq 1$) or ($j\neq 1$ and $k=j+1$)},\\
0 & \mbox{otherwise}.
\end{array}\right.
$$
\item {\bf Perfect Matching}. A perfect matching on $\{1, \ldots, p\}$ ($p$ even) is just a pairwise group of the vertices, and then, in our case, can be imposed as:
$$
x_{j k} = \left\{\begin{array}{cl} 1 & \mbox{if $j$ is odd and $k=j+1$},\\
0 & \mbox{otherwise}.
\end{array}\right.
$$
\end{itemize}

In Figure \ref{graphs} we show the shapes of spanning subgraphs described above.

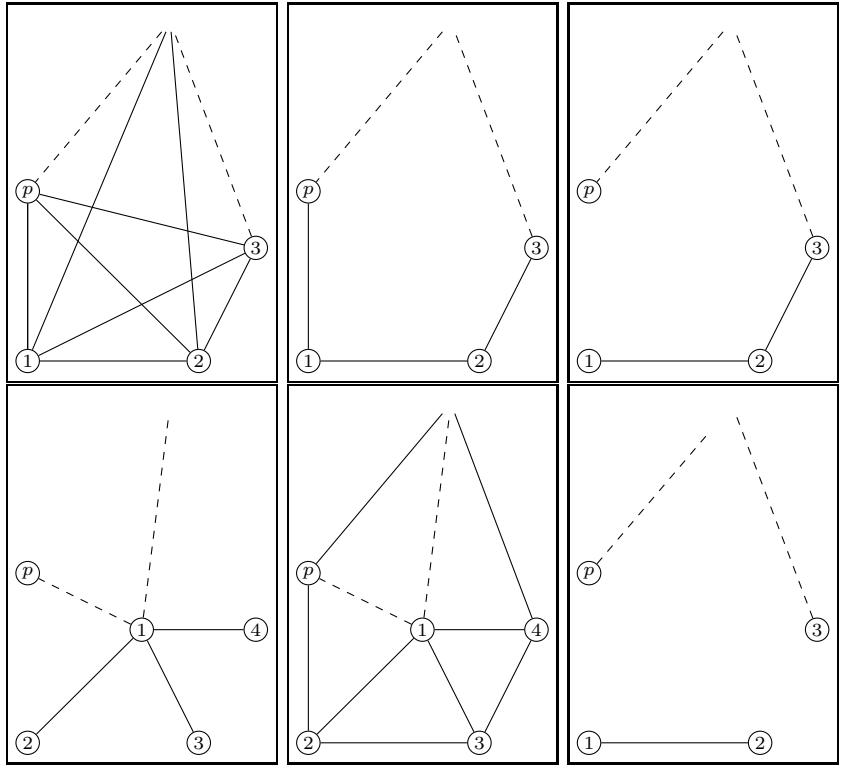
\begin{figure}[h]
\begin{center}
\fbox{\begin{tikzpicture}[scale=1.5]

  \node[draw, circle,inner sep=1pt](A) at (0,0) {\tiny $1$};
    \node[draw, circle,inner sep=1pt](B) at (1.5,0) {\tiny $2$};
      \node[draw, circle,inner sep=1pt](C) at (2,1) {\tiny $3$};
      \node at (1.25,3)(D) {};

          \node[draw, circle,inner sep=1pt](F) at (0,1.5) {\tiny $p$};
  \draw (A)--(B);
  \draw (A)--(C);
  \draw (A)--(D);
   \draw (A)--(F);
   \draw (B)--(C);
   \draw (B)--(D);
   \draw (B)--(F);
   \draw (C)--(F);
   \draw[dashed] (C)--(D);
   \draw[dashed] (D)--(F);
\draw (F)--(A);

\end{tikzpicture}}~\fbox{\begin{tikzpicture}[scale=1.5]

  \node[draw, circle,inner sep=1pt](A) at (0,0) {\tiny $1$};
    \node[draw, circle,inner sep=1pt](B) at (1.5,0) {\tiny $2$};
      \node[draw, circle,inner sep=1pt](C) at (2,1) {\tiny $3$};
      \node at (1.25,3)(D) {};

          \node[draw, circle,inner sep=1pt](F) at (0,1.5) {\tiny $p$};
  \draw (A)--(B);
        \draw (B)--(C);
            \draw[dashed] (C)--(D);
                        \draw[dashed] (D)--(F);
\draw (F)--(A);

\end{tikzpicture}}~\fbox{\begin{tikzpicture}[scale=1.5]

  \node[draw, circle,inner sep=1pt](A) at (0,0) {\tiny $1$};
    \node[draw, circle,inner sep=1pt](B) at (1.5,0) {\tiny $2$};
      \node[draw, circle,inner sep=1pt](C) at (2,1) {\tiny $3$};
      \node at (1.25,3)(D) {};

          \node[draw, circle,inner sep=1pt](F) at (0,1.5) {\tiny $p$};
  \draw (A)--(B);
        \draw (B)--(C);
            \draw[dashed] (C)--(D);
                        \draw[dashed] (D)--(F);

\end{tikzpicture}}

\fbox{\begin{tikzpicture}[scale=1.5]

  \node[draw, circle,inner sep=1pt](0) at (1,1) {\tiny $1$};
  \node[draw, circle,inner sep=1pt](A) at (0,0) {\tiny $2$};
    \node[draw, circle,inner sep=1pt](B) at (1.5,0) {\tiny $3$};
      \node[draw, circle,inner sep=1pt](C) at (2,1) {\tiny $4$};
      \node at (1.25,3)(D) {};

          \node[draw, circle,inner sep=1pt](F) at (0,1.5) {\tiny $p$};
           \draw (0)--(A);
  \draw (0)--(B);
        \draw (0)--(C);
            \draw[dashed] (0)--(D);
                        \draw[dashed] (0)--(F);

\end{tikzpicture}}~\fbox{\begin{tikzpicture}[scale=1.5]

  \node[draw, circle,inner sep=1pt](0) at (1,1) {\tiny $1$};
  \node[draw, circle,inner sep=1pt](A) at (0,0) {\tiny $2$};
    \node[draw, circle,inner sep=1pt](B) at (1.5,0) {\tiny $3$};
      \node[draw, circle,inner sep=1pt](C) at (2,1) {\tiny $4$};
      \node at (1.25,3)(D) {};
      \node at (1.1,2.8)(D1) {};

          \node[draw, circle,inner sep=1pt](F) at (0,1.5) {\tiny $p$};
           \draw (0)--(A);
  \draw (0)--(B);
        \draw (0)--(C);
            \draw[dashed] (0)--(D);
                        \draw[dashed] (0)--(F);
  \draw (A)--(B);
        \draw (B)--(C);
            \draw (C)--(D);
                        \draw (D)--(F);
 \draw (A)--(F);
\end{tikzpicture}}~\fbox{\begin{tikzpicture}[scale=1.5]

  \node[draw, circle,inner sep=1pt](A) at (0,0) {\tiny $1$};
    \node[draw, circle,inner sep=1pt](B) at (1.5,0) {\tiny $2$};
      \node[draw, circle,inner sep=1pt](C) at (2,1) {\tiny $3$};
      \node at (1.25,3)(D) {};

          \node[draw, circle,inner sep=1pt](F) at (0,1.5) {\tiny $p$};
  \draw (A)--(B);
            \draw[dashed] (C)--(D);
                        \draw[dashed] (D1)--(F);

\end{tikzpicture}}
\end{center}
\caption{Different shapes for the graphs considered. From top left to bottom right: Complete, Cycle, Line, Star, Ring-Star and Matching.\label{graphs}}
\end{figure}

\begin{example}\label{ex:1}
We consider a set of $15$ demand points on the plane, $\A=\{(0.34, 0.59)$, $(0.13, 0.9 ),  (0.67, 0.53),  (0.41, 0.03),  (0.36, 0.2 ),  (0.09, 0.1 ),  (0.29, 1.  ),  (0.68, 0.56), (0.08, 0.5 )$,\\ $(0.86, 0.71),  (0.66, 0.63), (0.87, 0.05),  (0.22, 0.44),  (0.22, 0.11),  (0.11, 0.53)\}$, $R_i=0.1$, for all $i=1, \ldots, 15$, $r=0.3$, $\omega_i=1$, for all $i$ and $\d$ the Euclidean distance on the plane. The solution to the classical MCLP for $p=6$ is drawn in Figure \ref{fig:mclp}. On the other hand, the solutions for the MCLPIF for the same number of facilities and shapes (Complete, Cycle, Line, Star, Ring-Star and Matching) are drawn in Figure \ref{fig:mclpif}. Demand points are drawn as circles while optimal facilities are drawn as asterisks. Segments between facilities represents links between them in the solution.

\begin{figure}[h]
\begin{center}
\fbox{
\begin{tikzpicture}[scale=1.5]

\begin{axis}[
hide x axis,
hide y axis,
axis equal,
]
\draw[draw=white!50.19607843137255!black,fill=white!50.19607843137255!black,opacity=0.2] (axis cs:0.280371354389903,0.514673321314653) circle (0.1);
\draw[draw=white!50.19607843137255!black,fill=white!50.19607843137255!black,opacity=0.2] (axis cs:0.386774725683932,0.11609031562243) circle (0.1);
\draw[draw=white!50.19607843137255!black,fill=white!50.19607843137255!black,opacity=0.2] (axis cs:0.211511196081541,0.9479501565164) circle (0.1);
\draw[draw=white!50.19607843137255!black,fill=white!50.19607843137255!black,opacity=0.2] (axis cs:0.663245268829768,0.573915554172247) circle (0.1);
\draw[draw=white!50.19607843137255!black,fill=white!50.19607843137255!black,opacity=0.2] (axis cs:0.157274650022056,0.114198708580954) circle (0.1);
\draw[draw=white!50.19607843137255!black,fill=white!50.19607843137255!black,opacity=0.2] (axis cs:0.105794816947374,0.511440295776329) circle (0.1);
\addplot [semithick, black, mark=asterisk, mark size=2, mark options={solid}, only marks]
table {%
0.280371354389903 0.514673321314653
};
\addplot [semithick, black, mark=*, mark size=1, mark options={solid}, only marks]
table {%
0.34 0.59
};
\addplot [semithick, black, mark=*, mark size=1, mark options={solid}, only marks]
table {%
0.22 0.44
};
\addplot [semithick, black, mark=asterisk, mark size=2, mark options={solid}, only marks]
table {%
0.386774725683932 0.11609031562243
};
\addplot [semithick, black, mark=*, mark size=1, mark options={solid}, only marks]
table {%
0.41 0.03
};
\addplot [semithick, black, mark=*, mark size=1, mark options={solid}, only marks]
table {%
0.36 0.2
};
\addplot [semithick, black, mark=asterisk, mark size=2, mark options={solid}, only marks]
table {%
0.211511196081541 0.9479501565164
};
\addplot [semithick, black, mark=*, mark size=1, mark options={solid}, only marks]
table {%
0.13 0.9
};
\addplot [semithick, black, mark=*, mark size=1, mark options={solid}, only marks]
table {%
0.29 1
};
\addplot [semithick, black, mark=asterisk, mark size=2, mark options={solid}, only marks]
table {%
0.663245268829768 0.573915554172247
};
\addplot [semithick, black, mark=*, mark size=1, mark options={solid}, only marks]
table {%
0.67 0.53
};
\addplot [semithick, black, mark=*, mark size=1, mark options={solid}, only marks]
table {%
0.68 0.56
};
\addplot [semithick, black, mark=*, mark size=1, mark options={solid}, only marks]
table {%
0.66 0.63
};
\addplot [semithick, black, mark=asterisk, mark size=2, mark options={solid}, only marks]
table {%
0.157274650022056 0.114198708580954
};
\addplot [semithick, black, mark=*, mark size=1, mark options={solid}, only marks]
table {%
0.09 0.1
};
\addplot [semithick, black, mark=*, mark size=1, mark options={solid}, only marks]
table {%
0.22 0.11
};
\addplot [semithick, black, mark=asterisk, mark size=2, mark options={solid}, only marks]
table {%
0.105794816947374 0.511440295776329
};
\addplot [semithick, black, mark=*, mark size=1, mark options={solid}, only marks]
table {%
0.08 0.5
};
\addplot [semithick, black, mark=*, mark size=1, mark options={solid}, only marks]
table {%
0.11 0.53
};
\addplot [semithick, black, mark=*, mark size=1, mark options={solid}, only marks]
table {%
0.86 0.71
};
\addplot [semithick, black, mark=*, mark size=1, mark options={solid}, only marks]
table {%
0.87 0.05
};

\end{axis}

\end{tikzpicture}}
\caption{Solution to the MCLP for the data of Example \ref{ex:1}.\label{fig:mclp}}
\end{center}
\end{figure}
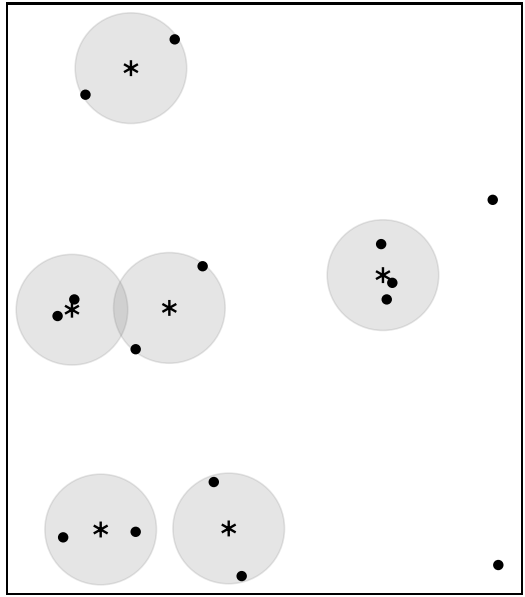

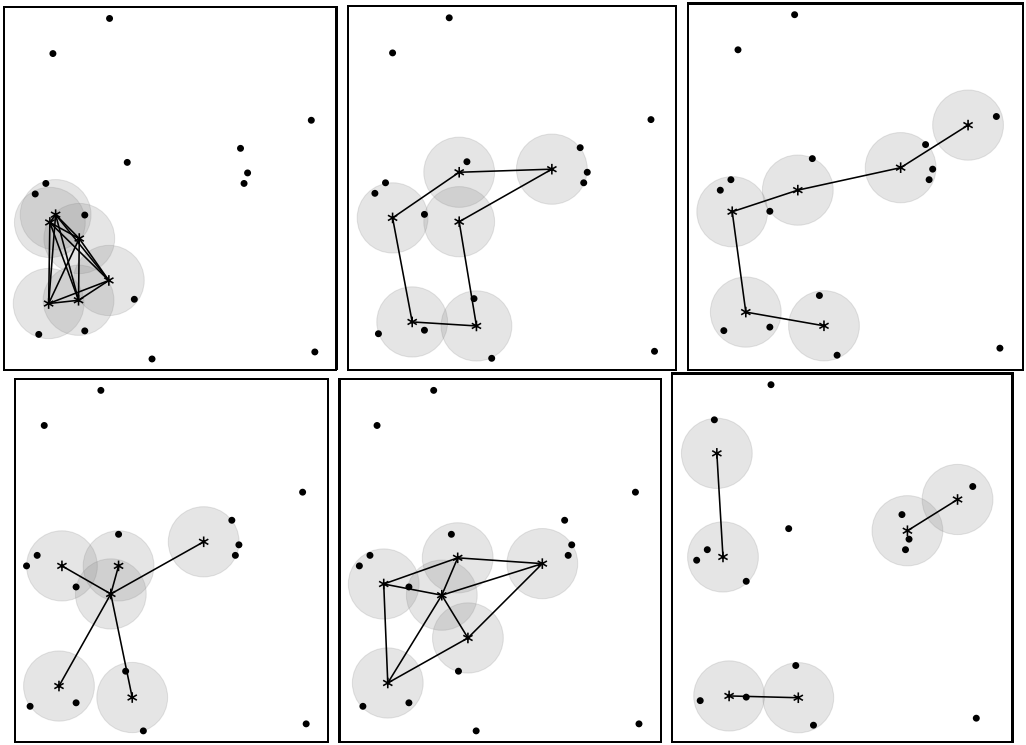
\begin{figure}[h]
\begin{center}
\fbox{\begin{tikzpicture}

\begin{axis}[
hide x axis,
hide y axis,
axis equal,
width=7cm,
height=7cm,
]
\draw[draw=white!50.19607843137255!black,fill=white!50.19607843137255!black,opacity=0.2] (axis cs:0.287990239341693,0.253638988832738) circle (0.1);
\draw[draw=white!50.19607843137255!black,fill=white!50.19607843137255!black,opacity=0.2] (axis cs:0.204422497693462,0.373070179911043) circle (0.1);
\draw[draw=white!50.19607843137255!black,fill=white!50.19607843137255!black,opacity=0.2] (axis cs:0.137635121007803,0.441071274100664) circle (0.1);
\draw[draw=white!50.19607843137255!black,fill=white!50.19607843137255!black,opacity=0.2] (axis cs:0.117942346607779,0.187884376091311) circle (0.1);
\draw[draw=white!50.19607843137255!black,fill=white!50.19607843137255!black,opacity=0.2] (axis cs:0.202378763151636,0.197197510440532) circle (0.1);
\draw[draw=white!50.19607843137255!black,fill=white!50.19607843137255!black,opacity=0.2] (axis cs:0.121352294779516,0.418887756819755) circle (0.1);
\addplot [semithick, black, mark=asterisk, mark size=2, mark options={solid}, only marks]
table {%
0.287990239341693 0.253638988832738
};
\addplot [semithick, black, mark=*, mark size=1, mark options={solid}, only marks]
table {%
0.36 0.2
};
\addplot [semithick, black, mark=asterisk, mark size=2, mark options={solid}, only marks]
table {%
0.204422497693462 0.373070179911043
};
\addplot [semithick, black, mark=*, mark size=1, mark options={solid}, only marks]
table {%
0.22 0.44
};
\addplot [semithick, black, mark=asterisk, mark size=2, mark options={solid}, only marks]
table {%
0.137635121007803 0.441071274100664
};
\addplot [semithick, black, mark=*, mark size=1, mark options={solid}, only marks]
table {%
0.11 0.53
};
\addplot [semithick, black, mark=asterisk, mark size=2, mark options={solid}, only marks]
table {%
0.117942346607779 0.187884376091311
};
\addplot [semithick, black, mark=*, mark size=1, mark options={solid}, only marks]
table {%
0.09 0.1
};
\addplot [semithick, black, mark=asterisk, mark size=2, mark options={solid}, only marks]
table {%
0.202378763151636 0.197197510440532
};
\addplot [semithick, black, mark=*, mark size=1, mark options={solid}, only marks]
table {%
0.22 0.11
};
\addplot [semithick, black, mark=asterisk, mark size=2, mark options={solid}, only marks]
table {%
0.121352294779516 0.418887756819755
};
\addplot [semithick, black, mark=*, mark size=1, mark options={solid}, only marks]
table {%
0.08 0.5
};
\addplot [semithick, black, mark=*, mark size=1, mark options={solid}, only marks]
table {%
0.34 0.59
};
\addplot [semithick, black, mark=*, mark size=1, mark options={solid}, only marks]
table {%
0.13 0.9
};
\addplot [semithick, black, mark=*, mark size=1, mark options={solid}, only marks]
table {%
0.67 0.53
};
\addplot [semithick, black, mark=*, mark size=1, mark options={solid}, only marks]
table {%
0.41 0.03
};
\addplot [semithick, black, mark=*, mark size=1, mark options={solid}, only marks]
table {%
0.29 1
};
\addplot [semithick, black, mark=*, mark size=1, mark options={solid}, only marks]
table {%
0.68 0.56
};
\addplot [semithick, black, mark=*, mark size=1, mark options={solid}, only marks]
table {%
0.86 0.71
};
\addplot [semithick, black, mark=*, mark size=1, mark options={solid}, only marks]
table {%
0.66 0.63
};
\addplot [semithick, black, mark=*, mark size=1, mark options={solid}, only marks]
table {%
0.87 0.05
};
\addplot [semithick, black]
table {%
0.287990239341693 0.253638988832738
0.204422497693462 0.373070179911043
};
\addplot [semithick, black]
table {%
0.287990239341693 0.253638988832738
0.137635121007803 0.441071274100664
};
\addplot [semithick, black]
table {%
0.287990239341693 0.253638988832738
0.117942346607779 0.187884376091311
};
\addplot [semithick, black]
table {%
0.287990239341693 0.253638988832738
0.202378763151636 0.197197510440532
};
\addplot [semithick, black]
table {%
0.287990239341693 0.253638988832738
0.121352294779516 0.418887756819755
};
\addplot [semithick, black]
table {%
0.204422497693462 0.373070179911043
0.137635121007803 0.441071274100664
};
\addplot [semithick, black]
table {%
0.204422497693462 0.373070179911043
0.117942346607779 0.187884376091311
};
\addplot [semithick, black]
table {%
0.204422497693462 0.373070179911043
0.202378763151636 0.197197510440532
};
\addplot [semithick, black]
table {%
0.204422497693462 0.373070179911043
0.121352294779516 0.418887756819755
};
\addplot [semithick, black]
table {%
0.137635121007803 0.441071274100664
0.117942346607779 0.187884376091311
};
\addplot [semithick, black]
table {%
0.137635121007803 0.441071274100664
0.202378763151636 0.197197510440532
};
\addplot [semithick, black]
table {%
0.137635121007803 0.441071274100664
0.121352294779516 0.418887756819755
};
\addplot [semithick, black]
table {%
0.117942346607779 0.187884376091311
0.202378763151636 0.197197510440532
};
\addplot [semithick, black]
table {%
0.117942346607779 0.187884376091311
0.121352294779516 0.418887756819755
};
\addplot [semithick, black]
table {%
0.202378763151636 0.197197510440532
0.121352294779516 0.418887756819755
};
\end{axis}

\end{tikzpicture}}~\fbox{\begin{tikzpicture}

\begin{axis}[
hide x axis,
hide y axis,
axis equal,
width=7cm,
height=7cm,
]
\draw[draw=white!50.19607843137255!black,fill=white!50.19607843137255!black,opacity=0.2] (axis cs:0.185393920141695,0.133619705186065) circle (0.1);
\draw[draw=white!50.19607843137255!black,fill=white!50.19607843137255!black,opacity=0.2] (axis cs:0.366928114184115,0.122193235918209) circle (0.1);
\draw[draw=white!50.19607843137255!black,fill=white!50.19607843137255!black,opacity=0.2] (axis cs:0.318169906025243,0.419120270956092) circle (0.1);
\draw[draw=white!50.19607843137255!black,fill=white!50.19607843137255!black,opacity=0.2] (axis cs:0.579999999999703,0.568786566428223) circle (0.1);
\draw[draw=white!50.19607843137255!black,fill=white!50.19607843137255!black,opacity=0.2] (axis cs:0.318169906025243,0.56) circle (0.1);
\draw[draw=white!50.19607843137255!black,fill=white!50.19607843137255!black,opacity=0.2] (axis cs:0.13,0.43) circle (0.1);
\addplot [semithick, black, mark=asterisk, mark size=2, mark options={solid}, only marks]
table {%
0.185393920141695 0.133619705186065
};
\addplot [semithick, black, mark=*, mark size=1, mark options={solid}, only marks]
table {%
0.09 0.1
};
\addplot [semithick, black, mark=*, mark size=1, mark options={solid}, only marks]
table {%
0.22 0.11
};
\addplot [semithick, black, mark=asterisk, mark size=2, mark options={solid}, only marks]
table {%
0.366928114184115 0.122193235918209
};
\addplot [semithick, black, mark=*, mark size=1, mark options={solid}, only marks]
table {%
0.41 0.03
};
\addplot [semithick, black, mark=*, mark size=1, mark options={solid}, only marks]
table {%
0.36 0.2
};
\addplot [semithick, black, mark=asterisk, mark size=2, mark options={solid}, only marks]
table {%
0.318169906025243 0.419120270956092
};
\addplot [semithick, black, mark=*, mark size=1, mark options={solid}, only marks]
table {%
0.22 0.44
};
\addplot [semithick, black, mark=asterisk, mark size=2, mark options={solid}, only marks]
table {%
0.579999999999703 0.568786566428223
};
\addplot [semithick, black, mark=*, mark size=1, mark options={solid}, only marks]
table {%
0.67 0.53
};
\addplot [semithick, black, mark=*, mark size=1, mark options={solid}, only marks]
table {%
0.68 0.56
};
\addplot [semithick, black, mark=*, mark size=1, mark options={solid}, only marks]
table {%
0.66 0.63
};
\addplot [semithick, black, mark=asterisk, mark size=2, mark options={solid}, only marks]
table {%
0.318169906025243 0.56
};
\addplot [semithick, black, mark=*, mark size=1, mark options={solid}, only marks]
table {%
0.34 0.59
};
\addplot [semithick, black, mark=asterisk, mark size=2, mark options={solid}, only marks]
table {%
0.13 0.43
};
\addplot [semithick, black, mark=*, mark size=1, mark options={solid}, only marks]
table {%
0.08 0.5
};
\addplot [semithick, black, mark=*, mark size=1, mark options={solid}, only marks]
table {%
0.11 0.53
};
\addplot [semithick, black, mark=*, mark size=1, mark options={solid}, only marks]
table {%
0.13 0.9
};
\addplot [semithick, black, mark=*, mark size=1, mark options={solid}, only marks]
table {%
0.29 1
};
\addplot [semithick, black, mark=*, mark size=1, mark options={solid}, only marks]
table {%
0.86 0.71
};
\addplot [semithick, black, mark=*, mark size=1, mark options={solid}, only marks]
table {%
0.87 0.05
};
\addplot [semithick, black]
table {%
0.185393920141695 0.133619705186065
0.366928114184115 0.122193235918209
};
\addplot [semithick, black]
table {%
0.185393920141695 0.133619705186065
0.13 0.43
};
\addplot [semithick, black]
table {%
0.366928114184115 0.122193235918209
0.318169906025243 0.419120270956092
};
\addplot [semithick, black]
table {%
0.318169906025243 0.419120270956092
0.579999999999703 0.568786566428223
};
\addplot [semithick, black]
table {%
0.579999999999703 0.568786566428223
0.318169906025243 0.56
};
\addplot [semithick, black]
table {%
0.318169906025243 0.56
0.13 0.43
};
\end{axis}

\end{tikzpicture}}~\fbox{\begin{tikzpicture}

\begin{axis}[
hide x axis,
hide y axis,
axis equal,
width=7cm,
height=7cm,
]
\draw[draw=white!50.19607843137255!black,fill=white!50.19607843137255!black,opacity=0.2] (axis cs:0.372935989066296,0.113762036758821) circle (0.1);
\draw[draw=white!50.19607843137255!black,fill=white!50.19607843137255!black,opacity=0.2] (axis cs:0.152042523483557,0.152835343798881) circle (0.1);
\draw[draw=white!50.19607843137255!black,fill=white!50.19607843137255!black,opacity=0.2] (axis cs:0.113689478087571,0.438698063373599) circle (0.1);
\draw[draw=white!50.19607843137255!black,fill=white!50.19607843137255!black,opacity=0.2] (axis cs:0.299020628582345,0.500240274138384) circle (0.1);
\draw[draw=white!50.19607843137255!black,fill=white!50.19607843137255!black,opacity=0.2] (axis cs:0.589599427934272,0.564212615403383) circle (0.1);
\draw[draw=white!50.19607843137255!black,fill=white!50.19607843137255!black,opacity=0.2] (axis cs:0.78000952823604,0.685483324561502) circle (0.1);
\addplot [semithick, black, mark=asterisk, mark size=2, mark options={solid}, only marks]
table {%
0.372935989066296 0.113762036758821
};
\addplot [semithick, black, mark=*, mark size=1, mark options={solid}, only marks]
table {%
0.41 0.03
};
\addplot [semithick, black, mark=*, mark size=1, mark options={solid}, only marks]
table {%
0.36 0.2
};
\addplot [semithick, black, mark=asterisk, mark size=2, mark options={solid}, only marks]
table {%
0.152042523483557 0.152835343798881
};
\addplot [semithick, black, mark=*, mark size=1, mark options={solid}, only marks]
table {%
0.09 0.1
};
\addplot [semithick, black, mark=*, mark size=1, mark options={solid}, only marks]
table {%
0.22 0.11
};
\addplot [semithick, black, mark=asterisk, mark size=2, mark options={solid}, only marks]
table {%
0.113689478087571 0.438698063373599
};
\addplot [semithick, black, mark=*, mark size=1, mark options={solid}, only marks]
table {%
0.08 0.5
};
\addplot [semithick, black, mark=*, mark size=1, mark options={solid}, only marks]
table {%
0.11 0.53
};
\addplot [semithick, black, mark=asterisk, mark size=2, mark options={solid}, only marks]
table {%
0.299020628582345 0.500240274138384
};
\addplot [semithick, black, mark=*, mark size=1, mark options={solid}, only marks]
table {%
0.34 0.59
};
\addplot [semithick, black, mark=*, mark size=1, mark options={solid}, only marks]
table {%
0.22 0.44
};
\addplot [semithick, black, mark=asterisk, mark size=2, mark options={solid}, only marks]
table {%
0.589599427934272 0.564212615403383
};
\addplot [semithick, black, mark=*, mark size=1, mark options={solid}, only marks]
table {%
0.67 0.53
};
\addplot [semithick, black, mark=*, mark size=1, mark options={solid}, only marks]
table {%
0.68 0.56
};
\addplot [semithick, black, mark=*, mark size=1, mark options={solid}, only marks]
table {%
0.66 0.63
};
\addplot [semithick, black, mark=asterisk, mark size=2, mark options={solid}, only marks]
table {%
0.78000952823604 0.685483324561502
};
\addplot [semithick, black, mark=*, mark size=1, mark options={solid}, only marks]
table {%
0.86 0.71
};
\addplot [semithick, black, mark=*, mark size=1, mark options={solid}, only marks]
table {%
0.13 0.9
};
\addplot [semithick, black, mark=*, mark size=1, mark options={solid}, only marks]
table {%
0.29 1
};
\addplot [semithick, black, mark=*, mark size=1, mark options={solid}, only marks]
table {%
0.87 0.05
};
\addplot [semithick, black]
table {%
0.372935989066296 0.113762036758821
0.152042523483557 0.152835343798881
};
\addplot [semithick, black]
table {%
0.152042523483557 0.152835343798881
0.113689478087571 0.438698063373599
};
\addplot [semithick, black]
table {%
0.113689478087571 0.438698063373599
0.299020628582345 0.500240274138384
};
\addplot [semithick, black]
table {%
0.299020628582345 0.500240274138384
0.589599427934272 0.564212615403383
};
\addplot [semithick, black]
table {%
0.589599427934272 0.564212615403383
0.78000952823604 0.685483324561502
};
\end{axis}

\end{tikzpicture}}\\
\fbox{\begin{tikzpicture}

\begin{axis}[
hide x axis,
hide y axis,
axis equal,
width=7cm,
height=7cm,
]
\draw[draw=lightgray!66.92810457516339!black,fill=lightgray!66.92810457516339!black,opacity=0.2] (axis cs:0.318093694802922,0.420110923095455) circle (0.1);
\draw[draw=lightgray!66.92810457516339!black,fill=lightgray!66.92810457516339!black,opacity=0.2] (axis cs:0.378570964719379,0.125021845480973) circle (0.1);
\draw[draw=lightgray!66.92810457516339!black,fill=lightgray!66.92810457516339!black,opacity=0.2] (axis cs:0.171633391585052,0.157758613489018) circle (0.1);
\draw[draw=lightgray!66.92810457516339!black,fill=lightgray!66.92810457516339!black,opacity=0.2] (axis cs:0.580485286516187,0.568623823390223) circle (0.1);
\draw[draw=lightgray!66.92810457516339!black,fill=lightgray!66.92810457516339!black,opacity=0.2] (axis cs:0.34,0.5) circle (0.1);
\draw[draw=lightgray!66.92810457516339!black,fill=lightgray!66.92810457516339!black,opacity=0.2] (axis cs:0.18,0.5) circle (0.1);
\addplot [semithick, black, mark=asterisk, mark size=2, mark options={solid}, only marks, forget plot]
table {%
0.318093694802922 0.420110923095455
};
\addplot [semithick, black, mark=*, mark size=1, mark options={solid}, only marks, forget plot]
table {%
0.22 0.44
};
\addplot [semithick, black, mark=asterisk, mark size=2, mark options={solid}, only marks, forget plot]
table {%
0.378570964719379 0.125021845480973
};
\addplot [semithick, black, mark=*, mark size=1, mark options={solid}, only marks, forget plot]
table {%
0.41 0.03
};
\addplot [semithick, black, mark=*, mark size=1, mark options={solid}, only marks, forget plot]
table {%
0.36 0.2
};
\addplot [semithick, black, mark=asterisk, mark size=2, mark options={solid}, only marks, forget plot]
table {%
0.171633391585052 0.157758613489018
};
\addplot [semithick, black, mark=*, mark size=1, mark options={solid}, only marks, forget plot]
table {%
0.09 0.1
};
\addplot [semithick, black, mark=*, mark size=1, mark options={solid}, only marks, forget plot]
table {%
0.22 0.11
};
\addplot [semithick, black, mark=asterisk, mark size=2, mark options={solid}, only marks, forget plot]
table {%
0.580485286516187 0.568623823390223
};
\addplot [semithick, black, mark=*, mark size=1, mark options={solid}, only marks, forget plot]
table {%
0.67 0.53
};
\addplot [semithick, black, mark=*, mark size=1, mark options={solid}, only marks, forget plot]
table {%
0.68 0.56
};
\addplot [semithick, black, mark=*, mark size=1, mark options={solid}, only marks, forget plot]
table {%
0.66 0.63
};
\addplot [semithick, black, mark=asterisk, mark size=2, mark options={solid}, only marks, forget plot]
table {%
0.34 0.5
};
\addplot [semithick, black, mark=*, mark size=1, mark options={solid}, only marks, forget plot]
table {%
0.34 0.59
};
\addplot [semithick, black, mark=asterisk, mark size=2, mark options={solid}, only marks, forget plot]
table {%
0.18 0.5
};
\addplot [semithick, black, mark=*, mark size=1, mark options={solid}, only marks, forget plot]
table {%
0.08 0.5
};
\addplot [semithick, black, mark=*, mark size=1, mark options={solid}, only marks, forget plot]
table {%
0.11 0.53
};
\addplot [semithick, black, mark=*, mark size=1, mark options={solid}, only marks, forget plot]
table {%
0.13 0.9
};
\addplot [semithick, black, mark=*, mark size=1, mark options={solid}, only marks, forget plot]
table {%
0.29 1
};
\addplot [semithick, black, mark=*, mark size=1, mark options={solid}, only marks, forget plot]
table {%
0.86 0.71
};
\addplot [semithick, black, mark=*, mark size=1, mark options={solid}, only marks, forget plot]
table {%
0.87 0.05
};
\addplot [semithick, black, forget plot]
table {%
0.318093694802922 0.420110923095455
0.378570964719379 0.125021845480973
};
\addplot [semithick, black, forget plot]
table {%
0.318093694802922 0.420110923095455
0.171633391585052 0.157758613489018
};
\addplot [semithick, black, forget plot]
table {%
0.318093694802922 0.420110923095455
0.580485286516187 0.568623823390223
};
\addplot [semithick, black, forget plot]
table {%
0.318093694802922 0.420110923095455
0.34 0.5
};
\addplot [semithick, black, forget plot]
table {%
0.318093694802922 0.420110923095455
0.18 0.5
};
\end{axis}

\end{tikzpicture}}~\fbox{\begin{tikzpicture}

\begin{axis}[
hide x axis,
hide y axis,
axis equal,
width=7cm,
height=7cm,
]
\draw[draw=lightgray!66.92810457516339!black,fill=lightgray!66.92810457516339!black,opacity=0.2] (axis cs:0.386730685930675,0.294711894254353) circle (0.1);
\draw[draw=lightgray!66.92810457516339!black,fill=lightgray!66.92810457516339!black,opacity=0.2] (axis cs:0.596912861414317,0.50643440247977) circle (0.1);
\draw[draw=lightgray!66.92810457516339!black,fill=lightgray!66.92810457516339!black,opacity=0.2] (axis cs:0.357857737327102,0.522937223866027) circle (0.1);
\draw[draw=lightgray!66.92810457516339!black,fill=lightgray!66.92810457516339!black,opacity=0.2] (axis cs:0.148849986664724,0.448401960855271) circle (0.1);
\draw[draw=lightgray!66.92810457516339!black,fill=lightgray!66.92810457516339!black,opacity=0.2] (axis cs:0.160283797804816,0.166429447610702) circle (0.1);
\draw[draw=lightgray!66.92810457516339!black,fill=lightgray!66.92810457516339!black,opacity=0.2] (axis cs:0.312938730519447,0.416567360061099) circle (0.1);
\addplot [semithick, black, mark=asterisk, mark size=2, mark options={solid}, only marks, forget plot]
table {%
0.386730685930675 0.294711894254353
};
\addplot [semithick, black, mark=*, mark size=1, mark options={solid}, only marks, forget plot]
table {%
0.36 0.2
};
\addplot [semithick, black, mark=asterisk, mark size=2, mark options={solid}, only marks, forget plot]
table {%
0.596912861414317 0.50643440247977
};
\addplot [semithick, black, mark=*, mark size=1, mark options={solid}, only marks, forget plot]
table {%
0.67 0.53
};
\addplot [semithick, black, mark=*, mark size=1, mark options={solid}, only marks, forget plot]
table {%
0.68 0.56
};
\addplot [semithick, black, mark=asterisk, mark size=2, mark options={solid}, only marks, forget plot]
table {%
0.357857737327102 0.522937223866027
};
\addplot [semithick, black, mark=*, mark size=1, mark options={solid}, only marks, forget plot]
table {%
0.34 0.59
};
\addplot [semithick, black, mark=asterisk, mark size=2, mark options={solid}, only marks, forget plot]
table {%
0.148849986664724 0.448401960855271
};
\addplot [semithick, black, mark=*, mark size=1, mark options={solid}, only marks, forget plot]
table {%
0.08 0.5
};
\addplot [semithick, black, mark=*, mark size=1, mark options={solid}, only marks, forget plot]
table {%
0.11 0.53
};
\addplot [semithick, black, mark=asterisk, mark size=2, mark options={solid}, only marks, forget plot]
table {%
0.160283797804816 0.166429447610702
};
\addplot [semithick, black, mark=*, mark size=1, mark options={solid}, only marks, forget plot]
table {%
0.09 0.1
};
\addplot [semithick, black, mark=*, mark size=1, mark options={solid}, only marks, forget plot]
table {%
0.22 0.11
};
\addplot [semithick, black, mark=asterisk, mark size=2, mark options={solid}, only marks, forget plot]
table {%
0.312938730519447 0.416567360061099
};
\addplot [semithick, black, mark=*, mark size=1, mark options={solid}, only marks, forget plot]
table {%
0.22 0.44
};
\addplot [semithick, black, mark=*, mark size=1, mark options={solid}, only marks, forget plot]
table {%
0.13 0.9
};
\addplot [semithick, black, mark=*, mark size=1, mark options={solid}, only marks, forget plot]
table {%
0.41 0.03
};
\addplot [semithick, black, mark=*, mark size=1, mark options={solid}, only marks, forget plot]
table {%
0.29 1
};
\addplot [semithick, black, mark=*, mark size=1, mark options={solid}, only marks, forget plot]
table {%
0.86 0.71
};
\addplot [semithick, black, mark=*, mark size=1, mark options={solid}, only marks, forget plot]
table {%
0.66 0.63
};
\addplot [semithick, black, mark=*, mark size=1, mark options={solid}, only marks, forget plot]
table {%
0.87 0.05
};
\addplot [semithick, black, forget plot]
table {%
0.386730685930675 0.294711894254353
0.596912861414317 0.50643440247977
};
\addplot [semithick, black, forget plot]
table {%
0.312938730519447 0.416567360061099
0.357857737327102 0.522937223866027
};
\addplot [semithick, black, forget plot]
table {%
0.312938730519447 0.416567360061099
0.148849986664724 0.448401960855271
};
\addplot [semithick, black, forget plot]
table {%
0.386730685930675 0.294711894254353
0.160283797804816 0.166429447610702
};
\addplot [semithick, black, forget plot]
table {%
0.386730685930675 0.294711894254353
0.312938730519447 0.416567360061099
};
\addplot [semithick, black, forget plot]
table {%
0.596912861414317 0.50643440247977
0.357857737327102 0.522937223866027
};
\addplot [semithick, black, forget plot]
table {%
0.596912861414317 0.50643440247977
0.312938730519447 0.416567360061099
};
\addplot [semithick, black, forget plot]
table {%
0.357857737327102 0.522937223866027
0.148849986664724 0.448401960855271
};
\addplot [semithick, black, forget plot]
table {%
0.148849986664724 0.448401960855271
0.160283797804816 0.166429447610702
};
\addplot [semithick, black, forget plot]
table {%
0.160283797804816 0.166429447610702
0.312938730519447 0.416567360061099
};
\end{axis}

\end{tikzpicture}}~\fbox{\begin{tikzpicture}

\begin{axis}[
hide x axis,
hide y axis,
axis equal,
width=7cm,
height=7cm,
]
\draw[draw=lightgray!66.92810457516339!black,fill=lightgray!66.92810457516339!black,opacity=0.2] (axis cs:0.171647831721368,0.113370239680867) circle (0.1);
\draw[draw=lightgray!66.92810457516339!black,fill=lightgray!66.92810457516339!black,opacity=0.2] (axis cs:0.367119309160633,0.108279276420335) circle (0.1);
\draw[draw=lightgray!66.92810457516339!black,fill=lightgray!66.92810457516339!black,opacity=0.2] (axis cs:0.675472119847715,0.583998523396568) circle (0.1);
\draw[draw=lightgray!66.92810457516339!black,fill=lightgray!66.92810457516339!black,opacity=0.2] (axis cs:0.817065776077334,0.673006906360159) circle (0.1);
\draw[draw=lightgray!66.92810457516339!black,fill=lightgray!66.92810457516339!black,opacity=0.2] (axis cs:0.136980861273969,0.804431693153517) circle (0.1);
\draw[draw=lightgray!66.92810457516339!black,fill=lightgray!66.92810457516339!black,opacity=0.2] (axis cs:0.154338109883223,0.509322710597004) circle (0.1);
\addplot [semithick, black, mark=asterisk, mark size=2, mark options={solid}, only marks, forget plot]
table {%
0.171647831721368 0.113370239680867
};
\addplot [semithick, black, mark=*, mark size=1, mark options={solid}, only marks, forget plot]
table {%
0.09 0.1
};
\addplot [semithick, black, mark=*, mark size=1, mark options={solid}, only marks, forget plot]
table {%
0.22 0.11
};
\addplot [semithick, black, mark=asterisk, mark size=2, mark options={solid}, only marks, forget plot]
table {%
0.367119309160633 0.108279276420335
};
\addplot [semithick, black, mark=*, mark size=1, mark options={solid}, only marks, forget plot]
table {%
0.41 0.03
};
\addplot [semithick, black, mark=*, mark size=1, mark options={solid}, only marks, forget plot]
table {%
0.36 0.2
};
\addplot [semithick, black, mark=asterisk, mark size=2, mark options={solid}, only marks, forget plot]
table {%
0.675472119847715 0.583998523396568
};
\addplot [semithick, black, mark=*, mark size=1, mark options={solid}, only marks, forget plot]
table {%
0.67 0.53
};
\addplot [semithick, black, mark=*, mark size=1, mark options={solid}, only marks, forget plot]
table {%
0.68 0.56
};
\addplot [semithick, black, mark=*, mark size=1, mark options={solid}, only marks, forget plot]
table {%
0.66 0.63
};
\addplot [semithick, black, mark=asterisk, mark size=2, mark options={solid}, only marks, forget plot]
table {%
0.817065776077334 0.673006906360159
};
\addplot [semithick, black, mark=*, mark size=1, mark options={solid}, only marks, forget plot]
table {%
0.86 0.71
};
\addplot [semithick, black, mark=asterisk, mark size=2, mark options={solid}, only marks, forget plot]
table {%
0.136980861273969 0.804431693153517
};
\addplot [semithick, black, mark=*, mark size=1, mark options={solid}, only marks, forget plot]
table {%
0.13 0.9
};
\addplot [semithick, black, mark=asterisk, mark size=2, mark options={solid}, only marks, forget plot]
table {%
0.154338109883223 0.509322710597004
};
\addplot [semithick, black, mark=*, mark size=1, mark options={solid}, only marks, forget plot]
table {%
0.08 0.5
};
\addplot [semithick, black, mark=*, mark size=1, mark options={solid}, only marks, forget plot]
table {%
0.22 0.44
};
\addplot [semithick, black, mark=*, mark size=1, mark options={solid}, only marks, forget plot]
table {%
0.11 0.53
};
\addplot [semithick, black, mark=*, mark size=1, mark options={solid}, only marks, forget plot]
table {%
0.34 0.59
};
\addplot [semithick, black, mark=*, mark size=1, mark options={solid}, only marks, forget plot]
table {%
0.29 1
};
\addplot [semithick, black, mark=*, mark size=1, mark options={solid}, only marks, forget plot]
table {%
0.87 0.05
};
\addplot [semithick, black, forget plot]
table {%
0.171647831721368 0.113370239680867
0.367119309160633 0.108279276420335
};
\addplot [semithick, black, forget plot]
table {%
0.675472119847715 0.583998523396568
0.817065776077334 0.673006906360159
};
\addplot [semithick, black, forget plot]
table {%
0.136980861273969 0.804431693153517
0.154338109883223 0.509322710597004
};
\end{axis}

\end{tikzpicture}}
\caption{Solution to the MCLPIF for the data of Example \ref{ex:1} and shapes (from top left to bottom right): Complete, Cycle, Line, Star, Ring-Star and Matching.\label{fig:mclpif}}
\end{center}
\end{figure}
\end{example}

Other graph structures can be modeled and incorporated to constraint \eqref{c1:4}, as for instance spanning trees. This can be done by using different modeling strategies. The classical way is by imposing that the subgraph has $p-1$ edges and the subtour elimination constraints (SEC), as:
\begin{align*}
\dsum_{j, k \in P: \atop k\neq j} x_{jk} &= p-1,\\
\dsum_{j, k \in S} x_{jk} &\leq |S|-1, \forall S \subseteq P, S\neq \emptyset
\end{align*}
 assuring that the obtained graph of facilities is a spanning tree. The SEC constraints can be efficiently separated, avoiding the use of exponentially many constraints in the formulation. Other \textit{compact} formulations are possible, as the one proposed by Miller, Tucker and Zemlim \cite{MTZ} or flow-based formulations as the following:
 \begin{align}
 \dsum_{j, k \in P: \atop k\neq j} x_{jk} &= p-1,\label{tree:0}\\
\dsum_{\ell \in P\atop \ell\neq 1} f_{j1\ell} &=1, &\forall j \in P, j\neq 1,\label{f:1}\\
\dsum_{\ell \in P\atop \ell\neq j} f_{j\ell j} &=1,& \forall j \in P, j\neq 1,\label{f:2}\\
\dsum_{\ell \in P} f_{jk\ell} -\dsum_{\ell \in P} f_{j\ell k} &=0, &\forall j, k\in P, j\neq 1, k\neq 1, \ell\neq j\label{f:3}\\
f_{jk\ell} &\leq x_{jk}, &\forall j, k, \ell \in P, j \neq 1.\label{f:4}
\end{align}
where $f_{jk\ell}$ indicates the amount of flow to sent from node $1$ to node $j$ using arc $(k,\ell)$ in the graph. \eqref{tree:0} indicates that the tree must have $p-1$ edges. Constraints \eqref{f:1} and \eqref{f:2} assure that when sending flow from node $1$ to node $j$, a first (different from $1$) and last (different from $j$) node in the network has to be used. Constraint \eqref{f:3} is a flow conservation constraint ensuring that flow entering a node must exit the node. Finally, constraint \eqref{f:4} avoid using nodes that do not belong to the network of the facilities. The above constraints \eqref{f:1}-\eqref{f:4} assure connectivity of the network of facilities.

Finally, we would like to highlight that the strategy proposed by Church in \cite{church1984} to reformulate the continuous MCLP as a discrete MCLP by means of the Circle Intersection Points is no longer valid for the MCLPIF as shown in the following simple counter example.
\begin{example}\label{ex:2}
Let us consider the set of five demand points on the plane $\A=\{(0,0), (1,0),$ $(3.25,0), (5,0)$, $(6,0)\}$, $R_i=0.5$ for $i=1, \ldots, 5$, $r=2.5$, $p=3$ and as $\S(G)$ a line. The CIP set is ${\rm CIP} = \A \cup \{(0.5,0), (5.5,0)\}$. However, the unique optimal solution to the problem to cover the five demand points is to locate the facilities in $X_1^*=(0.5,0)$, $X_2^*=(3,0)$ and $X_3^*=(5.5,0)$ as shown in Figure \ref{fg:ex2}.
\begin{figure}[h]
\begin{center}
\fbox{
\begin{tikzpicture}[scale=1.75]

\begin{axis}[
hide x axis,
hide y axis,
axis equal,
]
\draw[draw=lightgray!66.92810457516339!black,fill=lightgray!66.92810457516339!black,opacity=0.2] (axis cs:0.5,0) circle (0.5);
\draw[draw=lightgray!66.92810457516339!black,fill=lightgray!66.92810457516339!black,opacity=0.2] (axis cs:3,0) circle (0.5);
\draw[draw=lightgray!66.92810457516339!black,fill=lightgray!66.92810457516339!black,opacity=0.2] (axis cs:5.5,0) circle (0.5);
\addplot [semithick, black, mark=asterisk, mark size=2, mark options={solid}, only marks, forget plot]
table {%
0.5 0
};
\addplot [semithick, black, mark=*, mark size=1, mark options={solid}, only marks, forget plot]
table {%
0 0
};
\addplot [semithick, black, mark=*, mark size=1, mark options={solid}, only marks, forget plot]
table {%
1 0
};
\addplot [semithick, black, mark=asterisk, mark size=2, mark options={solid}, only marks, forget plot]
table {%
3 0
};
\addplot [semithick, black, mark=*, mark size=1, mark options={solid}, only marks, forget plot]
table {%
3.25 0
};
\addplot [semithick, black, mark=asterisk, mark size=2, mark options={solid}, only marks, forget plot]
table {%
5.5 0
};
\addplot [semithick, black, mark=*, mark size=1, mark options={solid}, only marks, forget plot]
table {%
5 0
};
\addplot [semithick, black, mark=*, mark size=1, mark options={solid}, only marks, forget plot]
table {%
6 0
};
\addplot [semithick, black, forget plot]
table {%
0.5 0
3 0
};
\addplot [semithick, black, forget plot]
table {%
3 0
5.5 0
};
\end{axis}

\end{tikzpicture}}
\caption{Unique optimal solution of the instance of Example \ref{ex:2}.\label{fg:ex2}}
\end{center}
\end{figure}
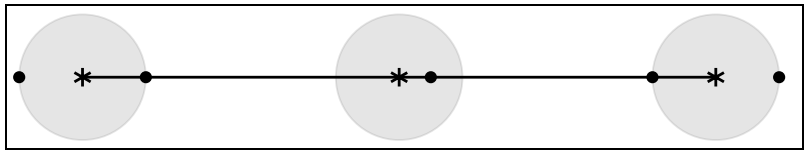
\end{example}

\section{An Integer Programming Formulation for the MCLPIF}\label{sec:3}

The model for the MCLPIF described in \eqref{NL} is a Mixed Integer Non Linear Programming formulation, which can be reformulated as a Mixed Integer Second Order Cone Optimization (MISOCO) problem and then, solved using any of the available off-the-shelf softwares (CPLEX, Gurobi, XPress, ...). However, the capability of MISOCO solvers is still far away of the efficiency of the routines for solving Mixed Integer Linear Programs. In this section we explore the way to solve the problems by using a MILP formulation involving only the binary decision of the MCLPIF, that is, the allocation variables ($z$) and the linking facilities variables ($x$).

The first observation that we draw is that the \textit{difficult} decisions of our problems are those obtained with the $z$ and the $x$-variables, since the continuous variables can be derived from them in polynomial time.
\begin{theorem}
Let $(\bar z, \bar x)$ be feasible solutions of the MCLPIF. Then, optimal positions for the facilities, $X_1, \ldots, X_p$, can be computed in polynomial time.
\end{theorem}
\begin{proof}
Observe that once the $z$ and the $x$ variables are fixed, the problem turns into:
          \begin{subequations}
    \makeatletter
        \def\@currentlabel{${\rm SP}$}
        \makeatother
       \label{SP}
        \renewcommand{\theequation}{${\rm SP}_{\arabic{equation}}$}
\begin{align}
	\xi(x,z):= \max & \dsum_{i\in N, j \in P: \atop \bar z_{ij}=1} \omega_i \label{p1:1}\\
	\mbox{s.t. } & \|a_i - X_j\| \leq R, \forall i\in N, j \in P: \bar z_{ij}=1,\label{p1:2}\\
			& \|X_j - X_k\| \leq r, \forall j, k \in P, j<k: \bar x_{jk}=1,\label{p1:3}\\
			& X_j \in \R^2, \forall j \in P.\label{p1:4}
\end{align}
\end{subequations}
which can be reformulated as a (continuous) SOC problem, and then, solved by interior point methods in polynomial time for any desired accuracy~\cite{NN94}. 
\end{proof}
Observe that, although \eqref{p1:1}-\eqref{p1:4} is a feasibility problem (the objective function is constant), if the optimal values for the $z$ and the $x$ variables are known, one may choose, among the multiple optimal solutions of the problem in the $X$-variables, for instance those minimizing the overall length of the graph just by considering the objective function:
$$
\dsum_{j,k \in P: \atop j<k} \d(X_j, X_k),
$$
obtaining a more \textit{compact} graph and still maintaining the SOC structure of the problem.

In what follows we analyze the feasible region of the problem \eqref{SP} above, and detail how this information can be exploited and incorporated to the $x$ and $z$ variables in order to project out the $X$-variables in \eqref{NL}.

Let $\bar z \in \{0,1\}^{n\times p}$ and $\bar x \in \{0,1\}^{p\times p}$ be feasible solutions of the MCLPIF, and denote by $C_j=\{i \in N: \bar z_{ij}=1\}$ and $K_j=\{k \in P: \bar x_{jk}=1\}$. Then, we get that:
\begin{itemize}
\item By constraint \eqref{p1:1}:
\begin{equation}\label{cov}
X_j \in \bigcap_{i \in C_j} \B_{R_i}(a_i), \forall j\in P,\tag{${\rm Cov}$}
\end{equation}
that is, $X_j$ is in the intersection of all $\d$-balls centered at the points covered by the facility the radia of them, $R_i$, for all $i\in C_j$.
\item By constraint \eqref{p1:2}:
\begin{equation}\label{link}
X_j \in \bigcap_{k\in K_j} \B_r(X_k), \forall j\in P,\tag{${\rm Link}$}
\end{equation}
 that is, the $j$-th facility must be reachable (at distance $r$) to all the facilities linked to it. 
\end{itemize}

The above conditions fully characterize all the feasible solutions of \eqref{p1:1}-\eqref{p1:4}. However, although \eqref{cov} is clearly determined from the input data (the demand points and the radia $R_i$), \eqref{link} depends on the coordinates of the $X$-variables, whose values are unknown. In what follows we derive a necessary a sufficient condition for the existence of feasible values of the $X$-variables in terms of the $z$ and $x$-variables that allows us to project them out from the nonlinear formulation.

Let us denote by $\oplus$ the Minkowski sum on $\R^d$ ($A \oplus B = \{a+b: a \in A, b \in B\}$, $\forall A, B \subset \R^d$). Also, given $\S(G)$, $\mathbf{C}=\{C_1, \ldots, C_p\}$ with $C_1, \ldots, C_p \subseteq N$ and $\mathbf{K}=\{K_1, \ldots, K_p\}$ with $K_1, \ldots, K_p \subseteq P$, we denote:
$$
z^\mathbf{C}_{ij}=\left\{\begin{array}{cl}
1 & \mbox{if $i\in C_j$},\\
0 & \mbox{otherwise}
\end{array}\right. \mbox{ and }
x^\mathbf{K}_{jk}=\left\{\begin{array}{cl}
1 & \mbox{if $j\in K_j$},\\
0 & \mbox{otherwise}
\end{array}\right.,
$$
for all $i \in N, j, k \in P$

\begin{theorem}\label{th:2}
Let $\mathbf{C}=\{C_1, \ldots, C_p\}$ with $C_1, \ldots, C_p \subseteq N$ be nonempty disjoint sets of demand points and $\mathbf{K}=\{K_1, \ldots, K_p\} \subset P$ defining the graph structure $\S(G)$. Then, the following conditions are equivalent:
\begin{enumerate}
\item The set \begin{equation}\label{bigintersect}
\mathcal{M}_j := \bigcap_{i \in C_j} \B_{R_i}(a_i) \cap \bigcap_{k\in K_j} \Big(\Big(\bigcap_{i \in C_k} \B_{R_i}(a_i) \Big) \oplus \B_r(0)\Big),
\end{equation}
is non empty, for all $j \in P$.
\item The exists $\mathbf{X}=(X_1, \ldots, X_p) \in \R^d$ such that $(z^\mathbf{C}, x^\mathbf{K}, \mathbf{X})$ is a feasible solution for the MCLPIF.
\end{enumerate}
\end{theorem}
\begin{proof}
Let us assume that $\mathbf{C}=\{C_1, \ldots, C_p\}$ and $\mathbf{K}=\{K_1, \ldots, K_p\}$ are given such that the $\M$-sets in \eqref{bigintersect} are nonempty. Then, we construct the $\bar z$ and $\bar x$- values as:
$$
\bar z_{ij} =z^\mathbf{C}_{ij} \quad \mbox{and} \quad \bar x_{jk} = x^\mathbf{K}_{jk}.
$$
Let us denote by $L_j = \bigcap_{i \in C_j} \B_{R_i}(a_i)$ for all $j\in P$.
%
Assume that there not exist $\mathbf{X}=(X_1, \ldots, X_p)$ such that $(x,z,\mathbf{X})$ is feasible for MCLPIF, i.e., for all $\mathbf{X}=(X_1, \ldots, X_p) \in L_1 \times \cdots \times L_p$ there exist $j_X\in P$ and $k_X \in K_{j_X}$ such that $\|X_{j_X}-X_{k_X}\| >r$, or equivalently, $X_{j_X} \not\in \B_r(X_{k_X})$. Thus, we have that for all $\mathbf{X} \in L_1 \times \cdots \times L_p$:
$$\
\Big(L_1 \times \bigcap_{k \in K_1} (L_k \oplus \B_r(0))\Big) \times \cdots \times \Big(L_p \times \bigcap_{k \in K_p} (L_k \oplus \B_r(0))\Big) = \emptyset
$$
Then, any of the sets $\M_j = L_j \cap \bigcap_{k \in K_j} (L_k \oplus \B_r(0))$ is empty, contradicting the non-emptyness of the $\M$-sets.

The other implication is straightforward.
\end{proof}

The above result allows us to reformulate \eqref{NL} using only the $z$ and $x$ variables. We denote, for any $\C=\{C_1, \ldots, C_{p}\}$ with $C_j \subseteq N$, for all $j\in P$ and $\K=\{K_1, \ldots, K_p\}$ with $K_j \subseteq P$ for all $j\in P$:
$$
\mathcal{M}_j(\mathbf{C}; \mathbf{K}) = \bigcap_{i \in C_j} \B_{R_i}(a_i) \cap \bigcap_{k\in K_j} \Big(\Big(\bigcap_{i \in C_k} \B_{R}(a_i) \Big) \oplus \B_r(0)\Big) \neq \emptyset
$$

\begin{cor}
A solution to the MCLPIF can be obtained by solving the following integer linear programming formulation:
\begin{align}
	\max & \dsum_{i\in N}  \dsum_{j \in P} \omega_i z_{ij} \label{c1:0}\\
	\mbox{s.t. } & \dsum_{j \in P} z_{ij} \leq 1, \forall i \in N, \label{c1:1}\\
	& \dsum_{i \in N} z_{ij} \geq 1, \forall j \in P, \label{c1:2}\\
			& \dsum_{i \in C_j} z_{ij} + \dsum_{k \in K_j} \dsum_{i \in C_k} z_{ik} + \dsum_{k \in K_j} x_{jk} \leq  |C_j| + \dsum_{k \in L} |C_k|  + |K_j| -1,\nonumber\\
			& \forall C_1, \ldots, C_{p} \subset N, K_1, \ldots, K_p \subseteq P,  \text{ with } \mathcal{M}_j(\C; \K) = \emptyset ,\label{c2:4}\\
			& x \in \mathcal{S}(G),\label{c2:4}\\
			& z_{ij}\in \{0,1\}, \forall i \in N, j \in P,\label{c2:5}\\
			& x_{jk}\in \{0,1\}, \forall j, k \in P. \label{c2:6}
\end{align}
\end{cor}
\begin{proof}
By Theorem \ref{th:2}, any feasible solution, in the $X$-variables, given the values of the $z$ and $x$-variables, verifies that 
$$
X_j \in  \mathcal{M}_j := \mathcal{M}_{j}(\C, \K) = \bigcap_{i \in C_j} \B_R(a_i)  \cap \bigcap_{k\in K_j} \Big(\Big(\bigcap_{i \in C_k} \B_{R}(a_i) \Big) \oplus \B_r(0)\Big)
$$
and also that the optimal solution of MCLPIF can be obtained by choosing adequately $X_j \in \mathcal{M}_j$ for $j\in P$. Thus, in order to be sure that the $z$ and $x$ variables induce non empty sets $\mathcal{M}_1, \ldots, \mathcal{M}_p$ where choosing $X_1, \ldots, X_p$ one has to impose that combinations of $C_1, \ldots, C_p$ and $K_1, \ldots, K_p$ inducing empty sets are not allowed.

In this way, constraint \eqref{c2:4} enforces that in case $\mathcal{M}_j = \emptyset$ the solution if no longer valid, and then, all other possibilities inducing non empty sets are feasible.

Once the clusters of points and links are obtained with the formulation above, one is assured that \eqref{p1:1}-\eqref{p1:3}, whose solutions are the desired coordinates of the centers.
\end{proof}

The above mathematical programming formulation reduces the search of an optimal solution of the MCLPIF to the $z$ and $x$ binary variables, avoiding continuous variables and non linear constraints, at the price of adding the exponentially many constraints in \eqref{c2:4}. Also, constraints in the shape of \eqref{c2:4} are added in case the sets $\mathcal{M}_j(\C; \K)$ are empty. Such a set is the intersection of the intersection of $\d$-balls and the intersections of Minkowski sums of intersection of $\d$-balls with $\d$-balls centered at the origin, that, although convex sets, they are not easy to handle. Actually, these sets have been recently named in the literature generalized Minkowski sets (see \cite{GMS19}). 

The following result allows us to reduce the \textit{emptyness tests} on the $\mathcal{M}$-sets. We denote by $\O_{i_1i_2} = \Big(\B_{R_{i_1}}(a_{i_1}) \cap \B_{R_{i_2}}(a_{i_2}) \Big) \oplus \B_r(0)$ for any $i_1, i_2 \in N$. Note that in case $|S| = 1$ the set $\O_{i_1i_1}$ reduces to $\B_{R_{i_1}+r}(a_{i_1})$.

\begin{lem}
Let $S \subseteq N$ then:
$$
\Big(\bigcap_{i \in S} \B_{R_i}(a_i)\Big) \oplus \B_{r}(0)  = \bigcap_{i_1,i_2 \in S} \O_{i_1i_2} 
$$
\end{lem}
\begin{proof}

First, observe that the intersection of the $R_i$-disks centered at the points indexed by $S$ is identical to the pairwise intersections in that index set, i.e. 
\begin{equation}\label{2by2}
\bigcap_{i \in S} \B_R(a_i) = \bigcap_{i_1, i_2 \in S} \Big(\B_{R_{i_1}}(a_{i_1}) \cap  \B_{R_{i_2}}(a_{i_2})\Big).
\end{equation}
Let us now check the identity in the result. On the one hand, let us assume that $z \in  \bigcap_{i \in S} \B_R(a_i) \oplus \B_{r}(0)$. Then, by \eqref{2by2}, there exist $ x\in \bigcap_{i_1, i_2 \in S} \left( \B_{R_{i_1}}(a_{i_1}) \cap  \B_{R_{i_2}}(a_{i_2}) \right)$ and $y \in \B_{r}(0)$ such that $z = x+y$. It implies that $x \in \B_{R_{i_1}}(a_{i_1}) \cap  \B_{R_{i_2}}(a_{i_2})$, $\forall i_1,i_2\in S$ and $z \in \Big(\B_{R_{i_1}}(a_{i_1}) \cap  \B_{R_{i_2}}(a_{i_2})\Big) \oplus \B_{r}(0) = \O_{i_1 i_2}$. Thus, $z \in \bigcap_{i_1, i_2 \in S} \O_{i_1 i_2}$.

On the other hand, let $z \in \bigcap_{i_1,i_2 \in S} \O_{i_1i_2}$. Then, $\forall i_1,i_2 \in S$, there exists $x_{i_1i_2},y_{i_1i_2}:  z = x_{i_1i_2}+y_{i_1i_2}$, implying that $x_{i,i'} \in \B_r(z)$, $\forall i,i' \in C_k, i\neq i'$. Then, $\B_r(z) \cap \left( \B_R(a_i) \cap \B_R(a_{i'}) \right) \neq \emptyset$ for all $i_1, i_2 \in S$. By Helly's Theorem, it assures that $\B_r(z) \cap \bigcap_{i_1,i_2 \in S} ( \B_R(a_i) \cap \B_R(a_{i'}) \neq \emptyset$. Then, there exists $x \in \B_r(z) \cap \left( \B_{R_{i_1}}(a_{i_1}) \cap \B_{R_{i_2}}(a_{i_2}) \right)\ \ \forall  i_1, i_2 \in S$. Thus, $z$ can be written as $z = x+(z-x)$, with $x \in \bigcap_{i_1, i_2 \in S} \left( \B_{R_{i_1}}(a_{i_1}) \cap \B_{R_{i_2}}(a_{i_2}) \right)$ and $(z-x) \in \B_{r}(0)$, being then $z \in \bigcap_{i, i' \in S} \left( \B_{R_{i_1}}(a_{i_1}) \cap  \B_{R_{i_2}}(a_{i_2}) \right) \oplus \B_{r}(0)$.

The case $|S| =1$ follows analogously.
\end{proof}

The following result allows us to replace the exponential number of constraints \eqref{c2:4} by a polynomial number of them (for fixed $d$).

\begin{theorem}\label{th:hellyballs0}
 $\C=\{C_1, \ldots, C_{p}\}$ with $C_j \subseteq N$, for all $j\in P$ and $\K=\{K_1, \ldots, K_p\}$ with $K_j \subseteq P$ for all $j\in P$. Then, $\mathcal{M}_j(\C;\K) = \emptyset$ if and only if any of the following intersections is empty:
 $$
 \bigcap_{i \in S^0} \B_{R_i} (a_{i}) \cap \bigcap_{i_1, i_2 \in S^1} \O_{i_1 i_2} = \emptyset,
 $$
 for $S^0, S^1 \subseteq N$ with $|S^0| + |S^1|=d+1$.

 \end{theorem}
 \begin{proof}
First observe that $\mathcal{M}_j(\C;\K)$ is a convex set since it is the intersections of convex sets ($\d$-balls and Minkowski sums of intersection of balls with a ball). Thus, by Helly's Theorem~\cite{helly}, only intersection of $(d+1)$-wise sets in the intersection is needed. In particular, observe that $\M_j(\C;\K)$ can be written as:
$$
\M_j(\C;\K) = \bigcap_{i \in C_j} \B_{R_i}(a_i) \cap \bigcap_{k \in K_j} \bigcap_{i_1, i_2 \in C_k} \O_{i_1i_2}.
$$
Then, choosing $(d+1)$ sets in the above intersection results in the stated results.
\end{proof}

\begin{cor}\label{th:hellyballs}
 Let $\C=\{C_1, \ldots, C_{p}\}$ with $C_j \subseteq N$, for all $j\in P$ and $\K=\{K_1, \ldots, K_p\}$ with $K_j \subseteq P$ for all $j\in P$. Then, in the planar case, $\mathcal{M}_j(\C;\K) = \emptyset$ if and only any of the following conditions is verified:
\begin{enumerate}
\item $\B_{R_{i_1}}(a_{i_1}) \cap \B_{R_{i_2}}(a_{i_2}) \cap \B_{R_{i_3}}(a_{i_3})  = \emptyset$, for all $i_1, i_2, i_3 \in C_j$.
\item $\B_{R_{i_1}}(a_{i_1}) \cap \B_{R_{i_2}}(a_{i_2}) \cap \O_{i_3i_4} = \emptyset$, for all $i_1, i_2 \in C_j$, $i_3, i_4 \in C_\ell$ for some $\ell \in C_k$ for some $k\in K_j$.
\item $\B_{R_{i_1}}(a_{i_1}) \cap  \O_{i_2i_3} \cap  \O_{i_4i_5} = \emptyset$, for all $i_1 \in C_j$, $i_2, i_3 \in C_{\ell_1}$, $i_4, i_5 \in C_{\ell_2}$ for $\ell_1, \ell_2 \in K_j$.
\item $ \O_{i_1i_2} \cap  \O_{i_3i_4} \cap  \O_{i_5i_6}= \emptyset$, for all $i_1, i_2\in C_{\ell_1}$, $i_3, i_4\in C_{\ell_2}$, $i_5, i_6 \in C_{\ell_3}$, for $\ell_1, \ell_2, \ell_3 \in K_j$.
 \end{enumerate}
 \end{cor}
\begin{proof}
In the planar case, $d+1=3$, and then choosing three sets in the above intersection results in the different statements of the corollary, i.e., intersections of three disks, intersections of two disks and one of the $\O$-sets, intersection of one disk and two $\O$-sets and the intersection of three $\O$-sets.
\end{proof}

In the following section we analyze how to handle $\O$-sets, and then $\M$-sets, for the planar case in order to determine whether a constraint of type \eqref{c2:4} must be added or not. 

\subsection{$\O$-sets}

In what follows we study the geometry of the $\O$-sets defined above in order to exploit it in an Integer Programming Formulation for the MCLPIF. Given $a_{i_1}, a_{i_2} \in \R^d$, this set is defined as:
$$
\O_{i_1i_2} = \Big(\B_{R_{i_1}}(a_{i_1}) \cap \B_{R_{i_2}}(a_{i_2}) \Big) \oplus \B_r(0).
$$
This set is the Minkowski sum of the intersection of two $\d$-balls centered at demand points $a_{i_1}$ and $a_{i_2}$ and a $\d$-ball centered at the origin. Since both sets are convex and bounded, $\O_{i_1i_2}$ is also bounded and convex. In Figure \ref{fig:o1} we illustrate the shape of this convex set in the planar case. There, we show the two points centers of the Euclidean balls (disks), $a_{i_1}$ and $a_{i_2}$, and the boundary of the intersection $\B_{R_{i_1}}(a_{i_1}) \cap \B_{R_{i_2}}(a_{i_2})$ that is drawn with a dashed line. The disks $\B_r(0)$ are then moved all around the points in the intersection of the disks. The border of $\O_{i_1i_2}$ is drawn with thick line in the picture. In Figure \ref{fig:o2} we show different shapes for the planar $\O$-sets.

\begin{figure}[h]
\begin{center}
\begin{subfigure}{0.5\textwidth}
\begin{center}
\begin{tikzpicture}[scale=7.4]

\coordinate (P1) at (0,0);
\coordinate (P2) at (1,0);

\node[above]  at (P1) {$a_{i_1}$};
\node[above]  at (P2) {$a_{i_2}$};

\fill (P1)  circle[radius=0.5pt];
\fill (P2)  circle[radius=0.5pt];
\coordinate (Q1) at (0.5,0.33166247903554);
\coordinate (Q2) at (0.5,-0.33166247903554);

\draw[dashed] (P1) +(-33.55730976192071:0.6) arc (-33.55730976192071:33.55730976192071:0.6);
\draw[dashed] (P2) +(146.4426902380793:0.6) arc (146.4426902380793:213.5573097619207:0.6);

\draw (0.5,-0.33166247903554) circle (0.1);

\draw (0.5,0.33166247903554) circle (0.1);

\draw (0.4,0) circle (0.1);

\draw[<->] (Q1) -- node {\tiny \hspace*{1.75pt} ${}_r$} (0.5,0.43166248);
\draw[<->] (P2) -- node[above] {\tiny ${}_R$} (0.42936609022290795, 0.1854101966249685);
\draw[<->] (P1) -- node[below] {\tiny ${}_R$} (0.552636596401731, -0.2336510053851903);

\draw[very thick] (P1) +(-33.55730976192071:0.7) arc (-33.55730976192071:33.55730976192071:0.7);
\draw[very thick] (P2) +(146.4426902380793:0.7) arc (146.4426902380793:213.5573097619207:0.7);

\draw[very thick] (Q1) + (33.55730976192071:0.1) arc (33.55730976192071:146.4426902380793:0.1);
\draw[very thick] (Q2) + (326.4426902380793:0.1) arc (326.4426902380793:213.5573097619207:0.1);

\end{tikzpicture}
\caption{Elements of a planar $\O$-set.\label{fig:o1}}
\end{center}
\end{subfigure}~\begin{subfigure}{0.5\textwidth}
\begin{center}
\begin{tikzpicture}[rotate=80]
\coordinate (P1) at (0.5,0);
\coordinate (P2) at (0,0.5);

\coordinate (Q1) at (0.91144508, 0.91144508);
\coordinate (Q2) at (-0.41144508, -0.41144508);

\draw[very thick] (P1) +(204.29539662504388:1.8) arc (204.29539662504388:65.70460337495611:1.8);
\draw[very thick] (P2) +(245.7046033749561:1.8) arc (-114.2953966250439:24.295396625043896:1.8);

\draw[very thick] (Q2) + (360-155.70460337495612:0.8) arc (360-155.70460337495612:360-114.29539662504389:0.8);
\draw[very thick] (Q1) + (24.295396625043896:0.8) arc  (24.295396625043896:65.70460337495611:0.8);

\end{tikzpicture}~\begin{tikzpicture}[scale=2,rotate=30]
\coordinate (P1) at (0.5,0);
\coordinate (P2) at (0,0.5);

\coordinate (Q1) at (0.75745063, 0.75745063);
\coordinate (Q2) at (-0.25745063, -0.25745063);

\draw[very thick] (P1) +(71.22756674437606:1) arc (71.22756674437606:198.77243325562392:1);
\draw[very thick] (P2) +(251.22756674437608:1) arc (-108.77243325562394:18.77243325562394:1);
\draw[very thick] (Q2) + (-108.77243325562397:0.2) arc (-108.77243325562397:-161.22756674437605:0.2);
\draw[very thick] (Q1) + (71.22756674437606:0.2) arc  (71.22756674437606:18.77243325562393:0.2);

\end{tikzpicture} \\
 \begin{tikzpicture}[scale=2.5,rotate=60]
\coordinate (P1) at (0.5,0);
\coordinate (P2) at (0,0.5);

\coordinate (Q1) at (0.38229337, 0.38229337);
\coordinate (Q2) at (0.11770663,0.11770663);

\draw[very thick] (P1) +(162.88660521741969:0.9) arc (162.88660521741969 : 107.11339478258033:0.9);
\draw[very thick] (P2) +(287.11339478258038:0.9) arc (287.1133947825803 : 342.8866052174197:0.9);
\draw[very thick] (Q2) + (287.1133947825803:0.5) arc (287.1133947825803 : 162.88660521741969:0.5);
\draw[very thick] (Q1) + (342.8866052174197:0.5) arc  (-17.113394782580315:107.11339478258033 : 0.5);

\end{tikzpicture}~\begin{tikzpicture}[scale=5]
\coordinate (P1) at (0.5,0);
\coordinate (P2) at (0,0.5);

\coordinate (Q1) at (0.38229337, 0.38229337);
\coordinate (Q2) at (0.11770663,0.11770663);

\draw[very thick] (P1) +(162.88660521741969:0.6) arc (162.88660521741969 : 107.11339478258033:0.6);
\draw[very thick] (P2) +(287.11339478258038:0.6) arc (287.1133947825803 : 342.8866052174197:0.6);
\draw[very thick] (Q2) + (287.1133947825803:0.2) arc (287.1133947825803 : 162.88660521741969:0.2);
\draw[very thick] (Q1) + (342.8866052174197:0.2) arc  (-17.113394782580315:107.11339478258033 : 0.2);

\end{tikzpicture}
\caption{Different shapes for planar $\O_{i_1i_2}$.\label{fig:o2}}
\end{center}
\label{fig:subim2}
\end{subfigure}
\caption{Planar $\O$-sets.\label{fig:o}}
\end{center}
\end{figure}
Note that $\O_{i_1i_2}$ is non empty provided that $\B_{R_{i_1}}(a_{i_1}) \cap \B_{R_{i_2}}(a_{i_2})$ is non empty. Also, observe that $\O_{i_1i_2} \subset \B_{R_{i_1}+r}(a_{i_1}) \cap \B_{R_{i_2}+r}(a_{i_2})$. This is clear since any element in $x \in \O_{i_1i_2}$ can be written as $x= z + y$ where $z \in \B_{R_{i_1}}(a_{i_1}) \cap \B_{R_{i_2}}(a_{i_2})$ and $y \in \B_r(0)$. Then, $\d(x,a_{i_1}) \leq \d(z + y, a_{i_1}) \leq \d(z,a_{i_1}) + \d(y,0) = R_{i_1}+r$ (analogously for $a_{i_2}$). In case $a_{i_1}=a_{i_2}$, $\O_{i_1i_2}$ coincides with $\B_{R_{i_1}+r}(a_{i_1})$, but since this case can be adequately fixed in a preprocessing phase, we assume that $a_{i_1}\neq a_{i_2}$. In such a case, $\O_{i_1i_2}$ is nothing but a \textit{smoothing} of $\B_{R_{i_1}+r}(a_{i_1}) \cap \B_{R_{i_2}+r}(a_{i_2})$ by two balls of radius $r$ in the two peaks of such a shape (see Figure \ref{fig:o3} where the boundary of $\O_{i_1i2}$ is drawn with a thick line and the boundary of the intersection $\B_{R_{i_1}+r}(a_{i_1}) \cap \B_{R_{i_2}+r}(a_{i_2})$ is drawn with a thiner gray line). Actually, as one can also observe from Figure \ref{fig:o3} that $\O_{i_1i_2}$ can be decomposed in three parts. On the one hand, the middle part in the picture coincide with $\B_{R_{i_1}+r}(a_{i_1}) \cap \B_{R_{i_2}+r}(a_{i_2})$ inside the strip delimited by the intersections points of the balls centered at $\{q_1, q_2\} = \partial \B_{R_{i_1}}(a_{i_1}) \cap \partial \B_{R_{i_2}}(a_{i_2})$ and radius $r$ (here, $\partial A$ denotes the boundary of the bounded set $A$). On the other hand, the two other parts of $\O_{i_1i_2}$ are just the balls with centers in $\{q_1, q_2\}$ and radius $r$. 
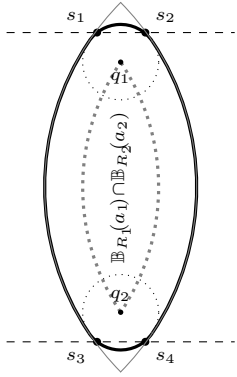
\begin{figure}[bth]
\begin{center}
\begin{tikzpicture}[scale=5]

\coordinate (P1) at (0,0);
\coordinate (P2) at (1,0);


\coordinate (Q1) at (0.5,0.33166247903554);
\coordinate (Q2) at (0.5,-0.33166247903554);

\coordinate (s1) at (0.43797,0.41009);
\coordinate (s2) at (0.565,0.410);

\fill (s1)  circle[radius=0.3pt];
\fill (s2)  circle[radius=0.3pt];
\node[above left] at (s1) {\tiny $s_{1}$};
\node[above right] at (s2) {\tiny $s_{2}$};

\coordinate (s3) at (0.43797,-0.41009);
\coordinate (s4) at (0.565,-0.410);

\fill (s3)  circle[radius=0.3pt];
\fill (s4)  circle[radius=0.3pt];
\node[ below left] at (s3) {\tiny $s_{3}$};
\node[ below right] at (s4) {\tiny $s_{4}$};


\draw[dotted] (Q2) circle (0.1);

\draw[dotted] (Q1) circle (0.1);

\draw[dashed] (0.2,0.41009)--(0.8,0.41009);
\draw[dashed] (0.2,-0.41009)--(0.8,-0.41009);

\draw[very thick] (P1) +(-33.55730976192071:0.7) arc (-33.55730976192071:33.55730976192071:0.7);
\draw[very thick] (P2) +(146.4426902380793:0.7) arc (146.4426902380793:213.5573097619207:0.7);

\draw[very thick, dotted, gray] (P1) +(-34:0.6) arc (-34:34:0.6);
\draw[very thick, dotted, gray] (P2) +(146:0.6) arc (146:213:0.6);

\draw[very thick] (Q1) + (33.55730976192071:0.1) arc (33.55730976192071:146.4426902380793:0.1);
\draw[very thick] (Q2) + (326.4426902380793:0.1) arc (326.4426902380793:213.5573097619207:0.1);

\draw[gray] (P1) +(-44.5:0.7) arc (-44.5:44.5:0.7);
\draw[gray] (P2) +(135.5:0.7) arc (135.5:224.5:0.7);

\fill (Q1)  circle[radius=0.2pt];
\fill (Q2)  circle[radius=0.2pt];
\node[below] at (Q1) {\tiny $q_{1}$};
\node[above] at (Q2) {\tiny $q_{2}$};

\node[rotate=90, font=\tiny] at (0.5,0) {$\B_{R_1}\!\!(a_1\!)\!\cap\!\B_{R_2}\!\!(a_2)$};

\end{tikzpicture}
\caption{Comparison of $\O_{i_1i_2}$ and $\B_{R+r}(a_{i_1}) \cap \B_{R+r}(a_{i_2})$.\label{fig:o3}}
\end{center}
\end{figure}

Thus, $\O_{i_1 i_2}$ can be written as the following (non disjoint) union:
$$
\O_{i_1 i_2} = \B_r(q_1) \cup \B_{r} (q_2) \cup \Big( \B_{R_{i_1}+r}(a_{i_1}) \cap  \B_{R_{i_2}+r}(a_{i_1}) \cap \mathcal{S}_{i_1i_2}\Big),
$$
where $\{q_1, q_2\} = \partial \B_{R_{i_1}}(a_{i_1}) \cap \partial \B_{R_{i_2}}(a_{i_2})$ and $\mathcal{S}_{i_1i_2} = \{z \in \R^d: \alpha_0 + \alpha^t z \leq 0 \leq \beta_0 + \beta^t z\}$ where $\alpha_0, \beta_0 \in \R, \alpha, \beta \in \R^d$ are the coefficients of the hyperplane passing through the points intersecting each of the balls centered at $q_1$ and $q_2$ and radius $r$ and $\partial \B_{R_{i_1}+r}(a_{i_1}) \cap \partial \B_{R_{i_2}+r}(a_{i_2})$ and such that $q_1, q_2 \in \mathcal{S}_{i_1i_2} $ (strip delimited by the dotted lines in Figure \ref{fig:o3}). 

Identifying the geometry of these sets will allow us to formulate, explicitly, and integer programming formulation for the MCLPIF with a polynomial number of constraints, as detailed in the next section.

With the above comments, we reformulate \eqref{c1:0}-\eqref{c2:6} in the planar case by means of linear constraints. In particular, by Theorem \ref{th:hellyballs} constraints \eqref{c2:4} can be replaced by the following set of constraints:

Let $i, i_1, \ldots, i_6 \in N$ and $j, k, k_1, k_2, k_3 \in P$.
\begin{itemize}
\item If $\B_{R_{i_1}}(a_{i_1}) \cap \B_{R_{i_2}}(a_{i_2}) \cap \B_{R_{i_3}}(a_{i_3})  = \emptyset$, then, two cases are possible:
	\begin{itemize}
	\item If $\B_{R_{\ell_1}}(a_{\ell_1}) \cap \B_{R_{\ell_2}}(a_{\ell_2}) = \emptyset$ for some $\ell_1, \ell_2 \in \{i_1, i_2, i_3\}$:
		\begin{equation}\label{r1}\tag{${\rm Int}_1$}
		 z_{{\ell_1}j}+z_{{\ell_2}j} \leq 1, \forall j \in P,
		\end{equation}
	that is, points $a_{\ell_1}$ and $a_{\ell_2}$ are incompatible to be covered by the same facility.
	\item If pairwise intersections are non empty but $\B_{R_{i_1}}(a_{i_1}) \cap \B_{R_{i_2}}(a_{i_2}) \cap \B_{R_{i_3}}(a_{i_3})  = \emptyset$:
		\begin{equation}\label{r2}\tag{${\rm Int}_2$}
		 z_{{i_1}j}+z_{{i_2}j}+z_{{i_3}j} \leq 2, \forall j \in P,
		\end{equation}
	which enforce the three points not to be covered by the same facility.
	\end{itemize}
\item If $\B_{R_{i_1}}(a_{i_1}) \cap \B_{R_{i_2}}(a_{i_2}) \cap \O_{i_3i_4} = \emptyset$, and assuming that the intersection of the two balls is nonempty, we have also two possibilities:
	\begin{itemize}
	\item If $\B_{R_{\ell}}(a_{\ell}) \cap \O_{i_3i_4} = \emptyset$ for some $\ell \in \{i_1, i_2\}$, then, we add the constraints:
		\begin{equation}\label{r3}\tag{${\rm Int}_3$}
		 z_{\ell j}+z_{{i_3}k}+z_{{i_4}k} + x_{jk} \leq 3, \forall j, k \in P,
		\end{equation}
	not allowing to simultaneously allocate $a_\ell$ to a facility ($j$) which is linked to other ($k$) to which $a_{i_3}$ and $a_{i_4}$ are allocated.
	\item If the ball and the $\O$-set intersect but the three wise intersection is empty:
		\begin{equation}\label{r4}\tag{${\rm Int}_4$}
		 z_{{i_1}j}+z_{{i_2}j}+z_{{i_3}k} + z_{{i_4}k} + x_{jk} \leq 4, \forall j, k \in P,
		\end{equation}
	indicating that $a_{i_1}$ and $a_{i_2}$ cannot be allocated to a (same) facility ($j$) and linked to another ($k$) that covers $a_{i_3}$ and $a_{i_4}$.
	\end{itemize}
\item $\B_{R_{i_1}}(a_{i_1}) \cap  \O_{i_2i_3} \cap  \O_{i_4i_5} = \emptyset$, but where the intersection of the ball and each of the two $\O$-sets is nonempty, again two options are possible:
	\begin{itemize}
	\item $\O_{i_2i_3} \cap  \O_{i_4i_5}=\emptyset$, and then, we add:
		\begin{equation}\label{r5}\tag{${\rm Int}_5$}
		z_{{i_2}k_1}+z_{{i_3}k_1}+z_{{i_4}k_2} + z_{{i_5}k_2} + x_{jk_1}+ x_{jk_2} \leq 5, \forall j, k_1, k_2 \in P.
		\end{equation}
	which does not allow a facility covering $a_{i_2}$ and $a_{i_3}$ ($k_1$) and a facility covering $a_{i_4}$ and $a_{i_5}$ ($k_2$) to share a common linked facility ($j$).
	\item In case the $\O$-sets intersect but the three-wise intersection is empty:
		\begin{equation}\label{r6}\tag{${\rm Int}_6$}
		z_{i_1j} + z_{{i_2}k_1}+z_{{i_3}k_1}+z_{{i_4}k_2} + z_{{i_5}k_2} + x_{jk_1}+ x_{jk_2} \leq 6, \forall j, k_1, k_2 \in P,
		\end{equation}
	that avoid linking the facility that covers $a_{i_1}$ simultaneously with the facilities that separately cover $a_{i_2}$ and $a_{i_3}$ ($k_1$) and $a_{i_4}$ and $a_{i_5}$ ($k_2$).
	\end{itemize}
\item Finally, if $\O_{i_1i_2} \cap  \O_{i_3i_4} \cap  \O_{i_5i_6}= \emptyset$ but pairwise intersections are nonempty, we add:
	\begin{equation}\label{r7}\tag{${\rm Int}_7$}
	z_{{i_1}k_1}+z_{{i_2}k_1}+z_{{i_3}k_2} + z_{{i_4}k_2} + z_{{i_5}k_3} +  z_{{i_6}k_3} + x_{jk_1}+ x_{jk_2} + x_{jk_3} \leq 8, \forall j, k_1, k_2, k_3 \in P.
	\end{equation}
which is the generalization to of \eqref{r5} to the case of three facilities ($k_1, k_2, k_3$) sharing a fourth common linked facility ($j$).
\end{itemize}

Summarizing the above comments we have the following result for the planar MCLPIF:
\begin{theorem}\label{formulation}
The planar MCLPIF can be solved by solving the following Integer Linear Programming problem:
          \begin{subequations}
    \makeatletter
        \def\@currentlabel{${\rm IP}$}
        \makeatother
       \label{IP}
        \renewcommand{\theequation}{${\rm IP}_{\arabic{equation}}$}
\begin{align}
	\max & \dsum_{i\in N}  \dsum_{j \in P} \omega_i z_{ij} \label{c1:0b}\\
	\mbox{s.t. } & \eqref{r1}-\eqref{r7},\\
	&\dsum_{j \in P} z_{ij} \leq 1, \forall i \in N, \label{c1:1b}\\
			& x \in \mathcal{S}(G),\nonumber\\
			& z_{ij}\in \{0,1\}, \forall i \in N, j \in P,\nonumber\\
			& x_{jk}\in \{0,1\}, \forall j, k \in P.\nonumber
\end{align}
\end{subequations}
\end{theorem}

\subsection{Computational Performance of \eqref{IP}\label{subsec:3.2}}

The above linear integer programming formulation, although a compact model for MCLPIF with a polynomial number of constraints, more precisely a worst case $O(p^3 n^6)$ constraints, it has still large number of constraints, that to be added, intersections of balls and $\O$-sets have to be computed.

We have run a series of computational experiments in order to evaluate the performance of formulation \eqref{IP}. We consider a set of small-medium instances based on the classical planar $50$-points dataset in \cite{EWC74} (normalized to the unit square $[0,1]\times[0,1]$) by randomly generating $5$ samples of sizes in $\{10, 20\}$. The number of centers to be located, $p$, ranges in $\{2, 6, 10\}$. We consider the same radia for all the demand points and ranging in $\{0.1, 0.2, 0.3\}$ and the limit distance between linked facilities $r \in \{0.2, 0.5\}$. The graph structures that we analyze are \texttt{Comp}, \texttt{Cycle}, \texttt{Line}, \texttt{Matching}, \texttt{Star} and \texttt{Star-Ring} (see Section \ref{subgraphs}).

The models were coded in Python 3.7 and we use as optimization solver Gurobi 9.0 in a MacBook Pro with a Core i5 CPU clocked at 2 GHz and 8GB of RAM memory. A time limit of $1$ hour was fixed for all the instances. 
\begin{table}[h]
	\centering
	\adjustbox{max totalheight=0.53\textheight}{\begin{tabular}{|c|cc|rrr|}\hline
		
		\texttt{Graph} & \texttt{n} & \texttt{p} & \texttt{IP\_Time} & \texttt{ContrGen\_Time} & \texttt{Tot\_Time} \\\hline
		
		\multirow{5}{*}{\texttt{Complete}} & \multirow{2}{*}{10} & 2     & 0.0082 & 44.7288 & 44.7370 \\
		&       & 6     & 0.9954 & 54.1249 & 55.1202 \\\cline{2-6}
		& \multirow{3}{*}{20} & 2     & 0.2196 & 1751.4293 & 1751.6489 \\
		&       & 6     & 44.9028 & 1904.8394 & 1949.7423 \\
		&       & 10    & 77.9067 & 2348.7582 & 2426.6649 \\\hline
		
		\multirow{5}{*}{\texttt{Cycle}} & \multirow{2}{*}{10} & 2     & 0.0085 & 45.9517 & 45.9601 \\
		&       & 6     & 0.1064 & 44.1848 & 44.2912 \\\cline{2-6}
		& \multirow{3}{*}{20} & 2     & 0.2333 & 1753.7321 & 1753.9654 \\
		&       & 6     & 6.7979 & 1813.7292 & 1820.5271 \\
		&       & 10    & 10.2933 & 1887.9462 & 1898.2395 \\\hline
		
		\multirow{5}{*}{\texttt{Line}} & \multirow{2}{*}{10} & 2     & 0.0082 & 45.4434 & 45.4517 \\
		&       & 6     & 0.0621 & 43.6978 & 43.7599 \\\cline{2-6}
		& \multirow{3}{*}{20}& 2     & 0.2205 & 1752.3162 & 1752.5367 \\
		&       & 6     & 2.6830 & 1815.9641 & 1818.6471 \\
		&       & 10    & 5.8076 & 1870.6905 & 1876.4980 \\\hline
		
		\multirow{5}{*}{\texttt{Matching}} & \multirow{2}{*}{10} & 2     & 0.0076 & 42.4407 & 42.4483 \\
		&       & 6     & 0.0187 & 43.9403 & 43.9590 \\\cline{2-6}
		& \multirow{3}{*}{20} & 2     & 0.2162 & 1752.6441 & 1752.8604 \\
		&       & 6     & 0.7954 & 1763.0113 & 1763.8068 \\
		&       & 10    & 1.3902 & 1859.2795 & 1860.6697 \\\hline
		
		\multirow{5}{*}{\texttt{Star}} & \multirow{2}{*}{10} & 2     & 0.0096 & 45.0242 & 45.0338 \\
		&       & 6     & 0.1943 & 44.5491 & 44.7434 \\\cline{2-6}
		& \multirow{3}{*}{20} & 2     & 0.2526 & 2309.5271 & 2309.7796 \\
		&       & 6     & 24.8045 & 1788.9676 & 1813.7721 \\
		&       & 10    & 85.0088 & 2042.7413 & 2127.7501 \\\hline
		
		\multirow{5}{*}{\texttt{Star-Ring}} & \multirow{2}{*}{10} & 2     & 0.0162 & 45.6987 & 45.7150 \\
		&       & 6     & 0.4755 & 47.5454 & 48.0209 \\\cline{2-6}
		& \multirow{3}{*}{20} & 2     & 1.0006 & 2085.5267 & 2086.5274 \\
		&       & 6     & 43.4228 & 1867.6105 & 1911.0333 \\
		&       & 10    & 76.1610 & 2043.4512 & 2119.6122 \\\hline
	\end{tabular}}%
		\caption{Average computational times for solving \eqref{IP}.\label{tab:ovals1}}
\end{table}%
The obtained results are shown in tables \ref{tab:ovals1} and \ref{tab:ovals2}. In both tables, the results are organized by graph structure (\texttt{Graph}), number of demand points (\texttt{n}) and number of centers to be located (\texttt{p}). In Table \ref{tab:ovals1} we report the required computational times in seconds for solving the IP problem with Gurobi (\texttt{IP\_Time}), the time consumed to generate the constraints of \eqref{IP} (column \texttt{ContrGen\_Time} ) and the total time required for both tasks plus the time of solving the continuous nonlinear problem \eqref{p1:1}-\eqref{p1:3} (column \texttt{Tot\_Time}).

One can observe that the times for solving the IP instance are reasonable, but the times needed to compute all the intersections for the constraints of the formulation are huge compared to the size of the instances. Concretely, $98.9\%$ of the total time is consumed computing the intersections. Thus, this approach is clearly computationally inefficient, fact that is justified with the results shown in Table \ref{tab:ovals2}. There, we show the average number of constraints involving only balls (\texttt{\#Ball\_Ctrs}), the constraints involving $\O$-sets (\texttt{\#O\_Ctrs}) and the overall number of constraints in the problem (\texttt{\#All\_Ctrs}). We also report in column \texttt{\%OoM}, the percentage of instances that flagged ``Out Of Memory'' when trying to solve the problem. From the results, we conclude that even for the smallest instances, the number of constraints to be added is huge, being the number of linear constraints needed to reflect the nonlinear nature of the problem is excessive. In particular, the percentage of $\O$-sets constraints with respect the overall number of constraints is $89.5\%$.
\begin{table}[h]
	\centering
	\adjustbox{max totalheight=0.5\textheight}{\begin{tabular}{|c|cc|rrrr|}\hline
		
		\texttt{Graph} & $n$ & $p$ & \texttt{\#Ball\_Ctrs} & \texttt{\#O\_Ctrs} & \texttt{\#All\_Ctrs} & \texttt{\%OoM} \\\hline
		
		\multirow{5}{*}{\texttt{Complete}} & \multirow{2}{*}{10} & 2     & 155   & 1444  & 1612  &0\% \\
		&       & 6     & 2144  & 300697 & 302857 &0\% \\\cline{2-7}
		& \multirow{3}{*}{20} & 2     & 1114  & 42335 & 43471 &0\% \\
		&       & 6     & 15155 & 5817910 & 5833091 & 3.33\% \\
		&       & 10    & 13557 & 10092470 & 10106057 & 50.00\% \\\hline
		
		\multirow{5}{*}{\texttt{Cycle}} & \multirow{2}{*}{10} & 2     & 155   & 1444  & 1612  &0\% \\
		&       & 6     & 806   & 32151 & 32973 &0\% \\\cline{2-7}
		& \multirow{3}{*}{20} & 2     & 1114  & 42233 & 43369 &0\% \\
		&       & 6     & 6159  & 719891 & 726076 &0\% \\
		&       & 10    & 10265 & 1258465 & 1268761 &0\% \\\hline
		
		\multirow{5}{*}{\texttt{Line}} & \multirow{2}{*}{10} & 2     & 155   & 1444  & 1612  &0\% \\
		&       & 6     & 693   & 20837 & 21546 &0\% \\\cline{2-7}
		& \multirow{3}{*}{20} & 2     & 1114  & 42450 & 43586 &0\% \\
		&       & 6     & 5220  & 356676 & 361922 &0\% \\
		&       & 10    & 9326  & 757653 & 767009 &0\% \\\hline
		
		\multirow{5}{*}{\texttt{Matching}} & \multirow{2}{*}{10} & 2     & 155   & 1444  & 1612  &0\% \\
		&       & 6     & 466   & 6832  & 7314  &0\% \\\cline{2-7}
		& \multirow{3}{*}{20} & 2     & 1114  & 41712 & 42848 &0\% \\
		&       & 6     & 3341  & 135866 & 139233 &0\% \\
		&       & 10    & 5569  & 261504 & 267102 &0\% \\\hline
		
		\multirow{5}{*}{\texttt{Star}} & \multirow{2}{*}{10} & 2     & 155   & 1444  & 1612  &0\% \\
		&       & 6     & 746   & 51232 & 51994 &0\% \\\cline{2-7}
		& \multirow{3}{*}{20} & 2     & 1114  & 48777 & 49913 &0\% \\
		&       & 6     & 5415  & 1065515 & 1070956 &0\% \\
		&       & 10    & 9623  & 4405466 & 4415119 & 6.67\% \\\hline
		
		\multirow{5}{*}{\texttt{Star-Ring}} & \multirow{2}{*}{10} & 2     & 212   & 5090  & 5315  &0\% \\
		&       & 6     & 1357  & 117151 & 118524 &0\% \\\cline{2-7}
		& \multirow{3}{*}{20} & 2     & 1583  & 178208 & 179813 &0\% \\
		&       & 6     & 10274 & 2621104 & 2631404 &0\% \\
		&       & 10    & 17523 & 6203147 & 6220700 & 6.67\% \\\hline
	\end{tabular}}%
		\caption{Average number of constraints of \eqref{IP}.\label{tab:ovals2}}%
\end{table}%
 To avoid the inconvenience of generating such an amount of constraints, in the following section, we propose two branch-\&-cut strategies for solving MCLPIF which initially only use a part of the constraints in \eqref{IP} and generate the rest as needed, resulting in more efficient solution approach.

\section{Branch-\&-Cut Approaches for the MCLPIF}\label{sec:4}

In this section we propose two exact methodologies for solving the planar MCLPIF that, instead of incorporating initially all the constraints \eqref{r1}-\eqref{r7}, the are incorporated on-the-fly as needed. The different approaches differ on the incomplete formulation which is considered.

\subsection{Incomplete formulation 1}

On the one hand we consider the following relaxed version of our problem:
          \begin{subequations}
    \makeatletter
        \def\@currentlabel{${\rm INC}^1$}
        \makeatother
       \label{IC1}
        \renewcommand{\theequation}{${\rm INC}^1_{\arabic{equation}}$}
\begin{align}
	\max & \dsum_{i\in N}  \dsum_{j \in P} \omega_i z_{ij} \label{c3:0}\\
	\mbox{s.t. } & \eqref{r1}, \eqref{r2},\\
	&\dsum_{j \in P} z_{ij} \leq 1, \forall i \in N, \label{c3:1}\\
			& x \in \mathcal{S}(G),\nonumber\
			& z_{ij}\in \{0,1\}, \forall i \in N, j \in P,\nonumber\\
			& x_{jk}\in \{0,1\}, \forall j, k \in P. \nonumber
\end{align}
\end{subequations}

In the above formulation, only the covering constraints are considered. It implies that the customers are clustered such that a set of $p$ facilities can be built to satisfy the users' demands, but it is not assured that the facilities can be adequately interconnected within a given distance $r$. Clearly, this formulation is a relaxation of the planar MCLPIF since the requirements on the maximum distance between linked facilities are not imposed. However, in problems in which the radius $r$ is large, compared to the $R_i$'s, the constraints induced by the $\O$-sets are only a few, and it may be more appropriate to incorporate them as long as they are violated. In order to check of a solution $(\bar z, \bar x)$ is valid or not for the planar MCLPIF we use the following auxiliary problem:
\begin{align}
\rho(\bar z, \bar x) = \min \dsum_{j \in P} \dsum_{k \in K_j} q_{jk}\\
\mbox{s.t. } & \d(X_j,X_k) \leq r + q_{jk}, \forall j \in P, k \in K_j,\\
& \d(X_j,a_i) \leq R_i, \forall j \in P, i \in C_j,\\
& q_{jk} \geq 0, \forall j\in P, k\in K_j,\\
& X_1,\ldots, X_p \in \R^2.
\end{align}  
where $K_j=\{k \in P: \bar x_{jk}=1\}$ and $C_j=\{i \in N: \bar z_{ij}=1\}$. This problem can be reformulated (for $\ell_\tau$-norms or polyhedral norms) as a Second Order Cone Programming problem, and it can be efficiently solved by interior points methods, which are implemented in the most popular commercial solvers.

In case $\rho(\bar z, \bar x)=0$, then, the obtained solution is feasible for the MCLPIF. Otherwise, the following cut is added to the incomplete formulation:
\begin{equation}\label{cuts}
\dsum_{j \in P} \dsum_{i \in C_j} z_{ij} + \dsum_{j \in P} \dsum_{k \in K_j} x_{jk} \leq \dsum_{j\in P} ( |C_j| + |K_j|) - 1.
\end{equation}

The advantage of this incomplete formulation is that only a few constraints (compared to \eqref{r3}-\eqref{r7}) are incorporated to the problem, and also the checking of empty intersection of balls and $\O$-sets is avoided. In contrast, if many cuts in the form \eqref{cuts} have to be incorporated, the solution process may slow down.

Note that the above incomplete formulation (and the separation problem) is valid for any distance measure $\d$ on the plane, provided that one is able to generate constraints \eqref{r1} and \eqref{r2}, i.e., checking intersection of two and three $\d$-balls.

\subsection{Incomplete formulation 2}

The second formulation consist of considering a relaxed version of the sets $\M_j$, as stated in the following result, replacing constraints \eqref{r3}-\eqref{r6} by using the following property already proved above:
$$
\M_j \subset \L_j := \bigcap_{i \in C_j} \B_{R_i}(a_i) \cap \bigcap_{k\in K_j} \Big(\bigcap_{i \in C_k} \B_{R_i+r}(a_i)\Big),
$$
for all $j\in P$.

Then, for $i, i_1, \ldots, i_6 \in N$, we consider the following constraints, imposing incompatibility of covered points and links based on $\L_j$ instead of $\M_j$:
\begin{itemize}
\item[] If $\B_{R_{i_1}}(a_{i_1}) \cap \B_{R_{i_2}+r}(a_{i_2}) = \emptyset$:
	\begin{equation}\label{r3p}\tag{${\rm Int}_3^\prime$}
	 z_{i_1 j}+z_{{i_2}k} + x_{jk} \leq 2, \forall j, k \in P,
	\end{equation}
\item[] If $\B_{R_{i_1}+r}(a_{i_1}) \cap \B_{R_{i_2}+r}(a_{i_2}) = \emptyset$:
	\begin{equation}\label{r4p}\tag{${\rm Int}_4^\prime$}
	 z_{{i_1}k_1}+z_{{i_2}k_2}+ x_{jk_1} + x_{jk_2} \leq 3, \forall j, k \in P,
	\end{equation}
\item[] If $\B_{R_{i_1}}(a_{i_1}) \cap \B_{R_{i_2}}(a_{i_2}) \cap \B_{R_{i_3}+r}(a_{i_3}) = \emptyset$:
	\begin{equation}\label{r5p}\tag{${\rm Int}_5^\prime$}
	z_{i_1j} + z_{i_2 j} + z_{i_3 k} + x_{jk} \leq 3, \forall j, k \in P.
	\end{equation}
\item[] If $\B_{R_{i_1}}(a_{i_1}) \cap \B_{R_{i_2}+r}(a_{i_2}) \cap \B_{R_{i_3}+r}(a_{i_3}) = \emptyset$:
	\begin{equation}\label{r6p}\tag{${\rm Int}_6\prime$}
	z_{i_1j} + z_{{i_2}k_1}+z_{{i_3}k_2} + x_{jk_1}+ x_{jk_2} \leq 4, \forall j, k_1, k_2 \in P,
	\end{equation}
\item[] If $\B_{R_{i_1}+r}(a_{i_1}) \cap \B_{R_{i_2}+r}(a_{i_2}) \cap \B_{R_{i_3}+r}(a_{i_3}) = \emptyset$:
	\begin{equation}\label{r7p}\tag{${\rm Int}_7^\prime$}
	z_{{i_1}k_1}+z_{{i_2}k_2}+z_{{i_3}k_3} + x_{jk_1}+ x_{jk_2}+x_{jk_3} \leq 5, \forall j, k_1, k_2, k_3 \in P.
	\end{equation}
\end{itemize}

Since $\M_j$ is strictly contained in $\L_j$, except in case $C_j$ is a singleton, the above constraints are just a relaxation of \eqref{r3}-\eqref{r7}. Thus, the following formulation is just an incomplete formulation for the MCLPIF:
          \begin{subequations}
    \makeatletter
        \def\@currentlabel{${\rm INC}^2$}
        \makeatother
       \label{IC2}
        \renewcommand{\theequation}{${\rm INC}^2_{\arabic{equation}}$}
\begin{align}
	\max & \dsum_{i\in N}  \dsum_{j \in P} \omega_i z_{ij} \label{c3:0}\\
	\mbox{s.t. } & \eqref{r1}, \eqref{r2},\\
	& \eqref{r3p}-\eqref{r7p},\\
	&\dsum_{j \in P} z_{ij} \leq 1, \forall i \in N, \label{c3:1}\\
			& x \in \mathcal{S}(G),\nonumber\\
			& z_{ij}\in \{0,1\}, \forall i \in N, j \in P, \nonumber\\
			& x_{jk}\in \{0,1\}, \forall j, k \in P. \nonumber
\end{align}
\end{subequations}
After solving the above problem, the test to check the validity of the formulation has to be run, i.e., solve the auxiliary problem to obtain $\rho(z,x)$, and add, if necessary, constraints \eqref{cuts}. The main advantage of this incomplete formulation compared to \eqref{IC1}, is that the representation of the feasible region is \textit{closer} to the exact representation of the problem, and then, in theory, a smaller number of constraints in the form \eqref{cuts} has to be added, implying solving a smaller number of IP problems.

Note that, as in the previous incomplete formulation , \eqref{IC2} is also valid for any ($\ell_\tau$ or polyhedral) distance measure $\d$ on the plane, but one has to be able to check intersections of two and three $\d$-balls (with different radia). In contrast, \eqref{IP} requires also checking intersections of $\O$-sets, which are more difficult to handle.

\section{Computational Experiments}\label{sec:5}

In this section we report the results of our computational experience in order to evaluate the performance of the different approaches for solving the MCLPIF. The experiments have been designed similarly to those reported in Section \ref{subsec:3.2} and run in the same computer. In particular, we use again the test instances based on the classical $50$-points dataset in \cite{EWC74} (normalized to the unit square $[0,1]\times[0,1]$) by randomly generating $5$ samples for sizes in $\{10, 20, 30, 40\}$, and the whole dataset of size $50$. The number of centers to be located, $p$, ranges in $\{2, 6, 10\}$. We consider, again, the same radia, $R_i$, for all the demand points and ranging in $\{0.1,0.2,0.3\}$ and the limit distance between linked facilities $r \in \{0.2,0.5\}$. The graph structures the we analyze are \texttt{Comp} (complete), \texttt{Cycle}, \texttt{Line}, \texttt{Matching}, \texttt{Star} and \texttt{Star-Ring}. In total, $2088$ instances were generated. A time limit of one hour was fixed for all the instances.

The detailed obtained results of our computational experiments as well as the used test instances are available through the \texttt{github} repository \url{github.com/vblancoOR/mclpif}.

In Figure \ref{fig:eilon} we illustrate optimal solutions for the $50$-points instance with $p=10$, $R_i=0.2$ and $r=0.5$ for the different graph structures.

\begin{figure}[h]
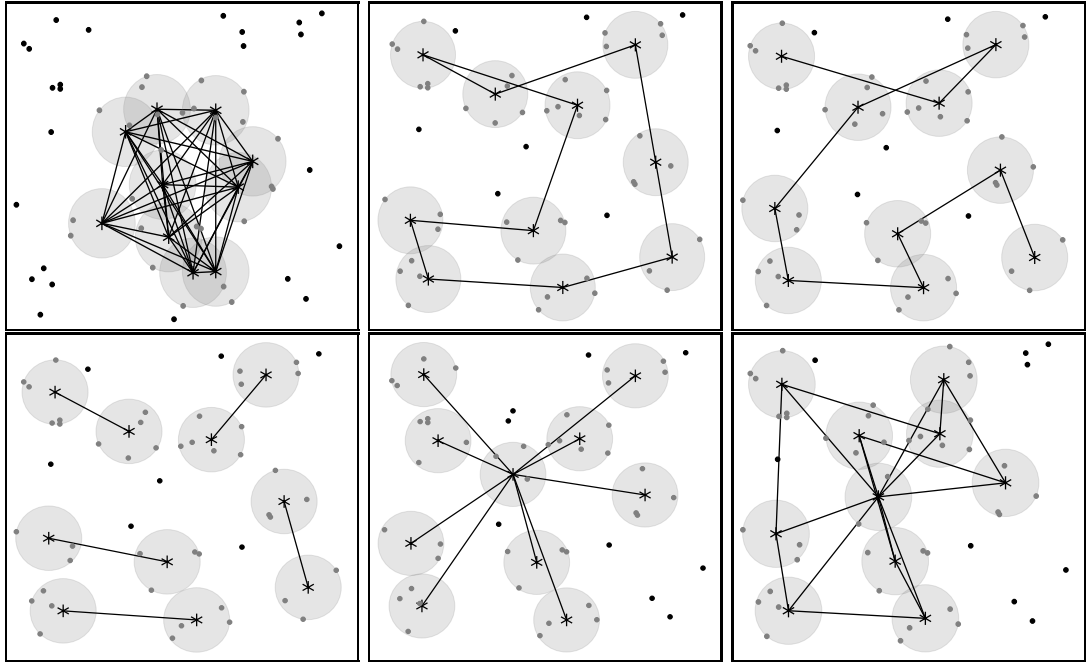

\begin{center}
\fbox{\resizebox{0.3\textwidth}{0.28\textwidth}{\input{Eilon50_Comp}}}~\fbox{\resizebox{0.3\textwidth}{0.28\textwidth}{\input{Eilon50_Cycle}}}~
\fbox{\resizebox{0.3\textwidth}{0.28\textwidth}{\input{Eilon50_Line}}}\\\fbox{\resizebox{0.3\textwidth}{0.28\textwidth}{\input{Eilon50_Matching}}}~
\fbox{\resizebox{0.3\textwidth}{0.28\textwidth}{\input{Eilon50_Star}}}~\fbox{\resizebox{0.3\textwidth}{0.28\textwidth}{\input{Eilon50_Star-Ring}}}
\caption{Optimal solutions for one configuration of parameters for the $50$-points instance ($p=10$, $R_i=0.2$ and $r=0.5$). From top left to bottom right: \texttt{Complete}, \texttt{Cycle}, \texttt{Line},\texttt{Matching}, \texttt{Star} and \texttt{Star-Ring}.\label{fig:eilon}}
\end{center}
\end{figure}

Since we have already reported the results of the approach based on formulation \eqref{IP} for the small instances (see Section \ref{subsec:3.2}), we analyze now the rest of the formulations, namely:
\begin{itemize}
\item \texttt{NL}: Compact Mixed Integer Non Linear formulation \eqref{NL}.
\item \texttt{INC}$_1$: Incomplete Integer Programming formulation \eqref{IC1}.
\item \texttt{INC}$_2$: Incomplete Integer Programming formulation \eqref{IC2}.
\end{itemize}
In tables \ref{t:times1} and \ref{t:times2} we show the average CPU times (in seconds) in columns \ref{NL}\texttt{\_Time}, \texttt{\ref{IC1}\_Time} and \texttt{\ref{IC2}\_Time} for each of the three approaches; and the percentage of unsolved instances (out of 30) in columns \ref{NL}\texttt{\_US}, \texttt{\ref{IC1}\_US} and \texttt{\ref{IC2}\_US}. The results are aggregated by graph type (\texttt{Graph}), number of demand points (\texttt{n}) and number of centers to be located (\texttt{p}). 

As can be observed, the incomplete formulations clearly outperform (in CPU time) the Mixed Integer Programming Formulation \eqref{NL}. In $83.25\%$ of the instances the CPU times required with the Incomplete formulations are smaller than the CPU times required by the MISOCO formulation. The overall average of the CPU times for \eqref{NL} was $616.41$ seconds, while for \eqref{IC1} 276.33 seconds, i.e., in average, the incomplete formulation required almost half of the time that the non linear formulation. Concerning the two incomplete formulations, one can observe that the results are similar, almost identical, in both CPU times and percentage of unsolved instances. As we will see with the rest of the tables (number of cuts and MIP Gaps), the results are also similar for these two incomplete formulations. Concerning the percentage of unsolved instances, the nonlinear formulation was not able to solve $292$ if the instances while the incomplete formulations only $139$. For instance, for the complete graph structure with $10$ facilities and $n = 50$ we obtained that none of the $30$ instances were solved within the time limit while incomplete formulations solved each of them in at most $96.83$ seconds. In contrast to formulation IP in which $98\%$ of the CPU time was consumed checking intersection and loading constraints to the models, in the incomplete formulations this percentage is approximately $50\%$ in average in all the instances.

In tables \ref{t:cuts1} and \ref{t:cuts2} we report the number of cuts required to solve the problem up to optimality in the branch-\&-cut approaches (columns \texttt{\ref{IC1}\_Cuts} and \texttt{\ref{IC2}\_Cuts}) and also the MIP Gaps obtained at the end of the time limit. As already mentioned, the number of cuts of the two branch-\&-cut approaches are almost identical. The MIP Gaps of the compact approach are also greater than those obtained with the branch-\&-cut approaches. In particular, in $87.35\%$ of the instances, the compact formulation obtained higher gaps than the incomplete formulations.

Finally, in Figure \ref{fig:times} we show the comparison on the performance of the CPU Times averaged by number of demand points for each of the graph types for both the Non Linear formulation and the incomplete IP formulation \eqref{IC1}, showing again that our approaches outperform the MINLP approach in most of the instances.
\begin{table}[H]
\centering
\adjustbox{max totalheight=0.9\textheight}{
	\begin{tabular}{|c|cr|rrr|rrr|}
		\hline
		\texttt{Graph} & $n$     & \multicolumn{1}{c|}{$p$} & \ref{NL}\texttt{\_Time} & \texttt{\ref{IC1}\_Time} & \texttt{\ref{IC2}\_Time} & \ref{NL}\texttt{\_US} & \ref{IC1}\texttt{\_US} & \ref{IC2}\texttt{\_US} \\\hline
		
		\multirow{14}{*}{\texttt{Comp}} & \multirow{2}[2]{*}{10} & 2     & 0.1342 & 0.0544 & 0.0567 &0\% &0\% &0\% \\
		&       & 6     & 2.6232 & 0.1151 & 0.1148 &0\% &0\% &0\% \\
		\cline{2-9}          & \multirow{3}[2]{*}{20} & 2     & 0.4766 & 0.2818 & 0.2807 &0\% &0\% &0\% \\
		&       & 6     & 29.0531 & 121.5217 & 121.5019 &0\% & 3.33\% & 3.33\% \\
		&       & 10    & 1549.7035 & 124.1093 & 123.8393 & 40.00\% & 3.33\% & 3.33\% \\
		\cline{2-9}          & \multirow{3}[2]{*}{30} & 2     & 2.6924 & 13.9745 & 13.9370 &0\% &0\% &0\% \\
		&       & 6     & 501.6179 & 304.3526 & 303.7910 &0\% & 6.67\% & 6.67\% \\
		&       & 10    & 3295.4696 & 246.1533 & 246.1096 & 90.00\% & 6.67\% & 6.67\% \\
		\cline{2-9}          & \multirow{3}[2]{*}{40} & 2     & 6.5542 & 115.7206 & 114.5961 &0\% &0\% &0\% \\
		&       & 6     & 2216.6469 & 576.8230 & 576.9015 & 43.33\% & 13.33\% & 13.33\% \\
		&       & 10    & 3360.1835 & 676.3061 & 678.1579 & 93.33\% & 13.33\% & 13.33\% \\
		\cline{2-9}          & \multirow{3}[2]{*}{50} & 2     & 10.5336 & 3.4213 & 3.5767 &0\% &0\% &0\% \\
		&       & 6     & 1321.6909 & 6.0239 & 6.2496 & 25.00\% &0\% &0\% \\
		&       & 10    & \texttt{TL} & 36.6441 & 36.9275 & 100.00\% &0\% &0\% \\
		\hline
		\multirow{14}{*}{\texttt{Cycle}} & \multirow{2}[2]{*}{10} & 2     & 0.1285 & 0.0552 & 0.0540 &0\% &0\% &0\% \\
		&       & 6     & 1.7382 & 0.5703 & 0.5697 &0\% &0\% &0\% \\
		\cline{2-9}          & \multirow{3}[2]{*}{20} & 2     & 0.4892 & 0.2753 & 0.2859 &0\% &0\% &0\% \\
		&       & 6     & 22.6150 & 272.4462 & 271.8050 &0\% & 6.67\% & 6.67\% \\
		&       & 10    & 141.4470 & 243.6100 & 243.7302 & 3.33\% & 6.67\% & 6.67\% \\
		\cline{2-9}          & \multirow{3}[2]{*}{30} & 2     & 2.6347 & 14.1975 & 14.0838 &0\% &0\% &0\% \\
		&       & 6     & 574.2565 & 964.3090 & 964.0870 & 6.67\% & 26.67\% & 26.67\% \\
		&       & 10    & 1212.6542 & 364.0570 & 363.7425 & 33.33\% & 10.00\% & 10.00\% \\
		\cline{2-9}          & \multirow{3}[2]{*}{40} & 2     & 6.5216 & 116.0356 & 115.2329 &0\% &0\% &0\% \\
		&       & 6     & 937.6188 & 730.6635 & 730.5040 & 16.67\% & 20.00\% & 20.00\% \\
		&       & 10    & 1337.4399 & 897.3624 & 896.9741 & 33.33\% & 20.00\% & 20.00\% \\
		\cline{2-9}          & \multirow{3}[2]{*}{50} & 2     & 10.6379 & 603.8272 & 604.0071 &0\% & 16.67\% & 16.67\% \\
		&       & 6     & 1735.0541 & 652.7675 & 653.8945 & 33.33\% & 16.67\% & 16.67\% \\
		&       & 10    & 1805.1871 & 1207.8929 & 1208.5188 & 50.00\% & 33.33\% & 33.33\% \\
		\hline
		\multirow{14}{*}{\texttt{Line}} & \multirow{2}[2]{*}{10} & 2     & 0.1293 & 0.0558 & 0.0617 &0\% &0\% &0\% \\
		&       & 6     & 1.0620 & 0.1106 & 0.1167 &0\% &0\% &0\% \\
		\cline{2-9}          & \multirow{3}[2]{*}{20} & 2     & 0.4907 & 0.2810 & 0.2912 &0\% &0\% &0\% \\
		&       & 6     & 22.5330 & 8.6143 & 8.5864 &0\% &0\% &0\% \\
		&       & 10    & 137.5967 & 1.9760 & 1.9692 & 3.33\% &0\% &0\% \\
		\cline{2-9}          & \multirow{3}[2]{*}{30} & 2     & 2.6478 & 14.0298 & 13.9174 &0\% &0\% &0\% \\
		&       & 6     & 652.1307 & 1031.5471 & 1031.0384 & 16.67\% & 26.67\% & 26.67\% \\
		&       & 10    & 1216.6928 & 243.6303 & 243.6503 & 33.33\% & 6.67\% & 6.67\% \\
		\cline{2-9}          & \multirow{3}[2]{*}{40} & 2     & 6.4847 & 116.3607 & 115.2736 &0\% &0\% &0\% \\
		&       & 6     & 899.5291 & 847.3092 & 847.2556 & 16.67\% & 23.33\% & 23.33\% \\
		&       & 10    & 1232.7082 & 537.2709 & 537.7941 & 33.33\% & 10.00\% & 10.00\% \\
		\cline{2-9}          & \multirow{3}[2]{*}{50} & 2     & 10.6999 & 604.2385 & 604.3965 &0\% & 16.67\% & 16.67\% \\
		&       & 6     & 1836.1251 & 609.0170 & 610.4558 & 50.00\% & 16.67\% & 16.67\% \\
		&       & 10    & 1284.3275 & 796.5951 & 806.4829 & 33.33\% &0\% &0\% \\
		\hline
	\end{tabular}%
	}
\caption{Averaged CPU times and percentages of unsolved instances for graphs \texttt{Comp}, \texttt{Cycle} and \texttt{Line}.\label{t:times1}}%
\end{table}%
\begin{table}[H]
\centering
\adjustbox{max totalheight=0.9\textheight}{
	\begin{tabular}{|c|cr|rrr|rrr|}
		\hline
		\texttt{Graph} & $n$     & \multicolumn{1}{c|}{$p$} & \multicolumn{1}{c}{\ref{NL}\texttt{\_Time}} & \multicolumn{1}{c}{\texttt{\ref{IC1}\_Time}} & \multicolumn{1}{c|}{\texttt{\ref{IC2}\_Time}} & \multicolumn{1}{c}{\ref{NL}\texttt{\_US}} & \multicolumn{1}{c}{\ref{IC1}\texttt{\_US}} & \multicolumn{1}{c|}{\ref{IC2}\texttt{\_US}} \\\hline
		
		\multirow{14}{*}{\texttt{Matching}} & \multirow{2}[2]{*}{10} & 2     & 0.1272 & 0.0535 & 0.0564 &0\% &0\% &0\% \\
		&       & 6     & 0.7815 & 0.0489 & 0.0490 &0\% &0\% &0\% \\
		\cline{2-9}          & \multirow{3}[2]{*}{20} & 2     & 0.4781 & 0.2782 & 0.2815 &0\% &0\% &0\% \\
		&       & 6     & 58.1331 & 0.3755 & 0.3777 &0\% &0\% &0\% \\
		&       & 10    & 258.4180 & 0.3421 & 0.3558 & 3.33\% &0\% &0\% \\
		\cline{2-9}          & \multirow{3}[2]{*}{30} & 2     & 2.6646 & 14.0095 & 13.9645 &0\% &0\% &0\% \\
		&       & 6     & 822.5527 & 136.4736 & 136.5727 & 16.67\% & 3.33\% & 3.33\% \\
		&       & 10    & 1203.3864 & 3.7853 & 3.8403 & 33.33\% &0\% &0\% \\
		\cline{2-9}          & \multirow{3}[2]{*}{40} & 2     & 6.4707 & 115.6995 & 114.5932 &0\% &0\% &0\% \\
		&       & 6     & 1357.2707 & 501.4939 & 501.4081 & 33.33\% & 13.33\% & 13.33\% \\
		&       & 10    & 1208.4601 & 189.9502 & 189.9281 & 33.33\% & 3.33\% & 3.33\% \\
		\cline{2-9}          & \multirow{3}[2]{*}{50} & 2     & 10.5842 & 603.9967 & 603.8510 &0\% & 16.67\% & 16.67\% \\
		&       & 6     & 1656.9411 & 319.2360 & 319.8650 & 33.33\% &0\% &0\% \\
		&       & 10    & 1228.5380 & 153.5888 & 155.1129 & 33.33\% &0\% &0\% \\
		\hline
		\multirow{14}{*}{\texttt{Star}} & \multirow{2}[2]{*}{10} & 2     & 0.1573 & 0.0571 & 0.0597 &0\% &0\% &0\% \\
		&       & 6     & 1.2605 & 122.0231 & 122.0039 &0\% & 3.33\% & 3.33\% \\
		\cline{2-9}          & \multirow{3}[2]{*}{20} & 2     & 0.4963 & 0.2842 & 0.2870 &0\% &0\% &0\% \\
		&       & 6     & 39.9333 & 242.0655 & 242.0854 &0\% & 6.67\% & 6.67\% \\
		&       & 10    & 945.8562 & 243.5264 & 243.6191 & 20.00\% & 6.67\% & 6.67\% \\
		\cline{2-9}          & \multirow{3}[2]{*}{30} & 2     & 2.6683 & 14.3389 & 14.1774 &0\% &0\% &0\% \\
		&       & 6     & 764.3177 & 130.0206 & 129.9031 & 20.00\% & 3.33\% & 3.33\% \\
		&       & 10    & 1985.6919 & 843.4682 & 843.3631 & 53.33\% & 23.33\% & 23.33\% \\
		\cline{2-9}          & \multirow{3}[2]{*}{40} & 2     & 6.4678 & 116.7553 & 115.2507 &0\% &0\% &0\% \\
		&       & 6     & 1122.4528 & 952.7522 & 951.6508 & 16.67\% & 23.33\% & 23.33\% \\
		&       & 10    & 2165.0432 & 876.6598 & 876.5971 & 60.00\% & 23.33\% & 23.33\% \\
		\cline{2-9}          & \multirow{3}[2]{*}{50} & 2     & 10.9528 & 604.4651 & 603.8411 &0\% & 16.67\% & 16.67\% \\
		&       & 6     & 1693.4531 & 607.4625 & 608.5091 & 33.33\% & 16.67\% & 16.67\% \\
		&       & 10    & 2401.2075 & 1276.8858 & 1277.4972 & 66.67\% & 33.33\% & 33.33\% \\
		\hline
		\multirow{14}{*}{\texttt{Star-Ring}} & \multirow{2}[2]{*}{10} & 2     & 0.1104 & 0.0593 & 0.0580 &0\% &0\% &0\% \\
		&       & 6     & 1.3273 & 2.4013 & 2.4025 &0\% &0\% &0\% \\
		\cline{2-9}          & \multirow{3}[2]{*}{20} & 2     & 0.3049 & 0.2880 & 0.2889 &0\% &0\% &0\% \\
		&       & 6     & 13.9281 & 190.3787 & 189.4114 &0\% & 3.33\% & 3.33\% \\
		&       & 10    & 231.6212 & 259.1514 & 258.7089 & 3.33\% & 6.67\% & 6.67\% \\
		\cline{2-9}          & \multirow{3}[2]{*}{30} & 2     & 0.7661 & 14.6793 & 14.6180 &0\% &0\% &0\% \\
		&       & 6     & 436.0929 & 965.7536 & 965.9031 & 6.67\% & 26.67\% & 26.67\% \\
		&       & 10    & 1865.2045 & 800.0823 & 799.9905 & 40.00\% & 20.00\% & 20.00\% \\
		\cline{2-9}          & \multirow{3}[2]{*}{40} & 2     & 1.8539 & 119.0502 & 118.0972 &0\% &0\% &0\% \\
		&       & 6     & 853.0198 & 741.5927 & 741.8315 & 13.33\% & 20.00\% & 20.00\% \\
		&       & 10    & 1889.9757 & 1036.1564 & 1036.8029 & 50.00\% & 26.67\% & 26.67\% \\
		\cline{2-9}          & \multirow{3}[2]{*}{50} & 2     & 5.9508 & 603.9291 & 603.9896 &0\% & 16.67\% & 16.67\% \\
		&       & 6     & 1434.6837 & 643.9400 & 644.1840 & 33.33\% & 16.67\% & 16.67\% \\
		&       & 10    & 1812.6836 & 676.1130 & 677.3516 & 50.00\% & 16.67\% & 16.67\% \\
		\hline
	\end{tabular}%
}
\caption{Averaged CPU times and percentages of unsolved instances for graphs \texttt{Matching}, \texttt{Star} and \texttt{Star-Ring}.\label{t:times2}}%
\end{table}%
\begin{table}[H]
	\centering
	\adjustbox{max totalheight=0.9\textheight}{
	\begin{tabular}{|c|cr|rr|rrr|}
		\hline
		\texttt{Graph} & $n$     & $p$ & \texttt{\#Cuts\_\eqref{IC1}} &\texttt{\#Cuts\_\eqref{IC2}} & \texttt{gap\_NL} & \texttt{gap\_\eqref{IC1}} & \texttt{gap\_\eqref{IC2}}\\
		\hline
		\multirow{14}{*}{\texttt{Comp}} & \multirow{2}[2]{*}{10} & 2     & 3     & 3     &0\% &0\% &0\% \\
		&       & 6     & 0     & 0     & 30.00\% & 30.00\% & 30.00\% \\
		\cline{2-8}          & \multirow{3}[2]{*}{20} & 2     & 8     & 8     &0\% &0\% &0\% \\
		&       & 6     & 1984  & 2009  &0\% &0\% &0\% \\
		&       & 10    & 704   & 700   & 25.24\% & 23.67\% & 23.67\% \\
		\cline{2-8}          & \multirow{3}[2]{*}{30} & 2     & 489   & 489   &0\% &0\% &0\% \\
		&       & 6     & 4199  & 4236  &0\% & 0.44\% & 0.44\% \\
		&       & 10    & 1509  & 1528  & 14.07\% & 0.48\% & 0.48\% \\
		\cline{2-8}          & \multirow{3}[2]{*}{40} & 2     & 1754  & 1754  &0\% &0\% &0\% \\
		&       & 6     & 7187  & 7235  & 3.95\% & 0.74\% & 0.74\% \\
		&       & 10    & 3895  & 3890  & 17.43\% & 0.94\% & 0.94\% \\
		\cline{2-8}          & \multirow{3}[2]{*}{50} & 2     & 7452  & 7512  &0\% & 0.62\% & 0.62\% \\
		&       & 6     & 7012  & 6890  & 13.19\% & 0.46\% & 0.46\% \\
		&       & 10    & 154   & 154   & 32.99\% &0\% &0\% \\
		\hline
		\multirow{14}{*}{\texttt{Cycle}} & \multirow{2}[2]{*}{10} & 2     & 3     & 3     &0\% &0\% &0\% \\
		&       & 6     & 25    & 25    & 13.33\% & 13.33\% & 13.33\% \\
		\cline{2-8}          & \multirow{3}[2]{*}{20} & 2     & 8     & 8     &0\% &0\% &0\% \\
		&       & 6     & 4657  & 4736  &0\% &0\% &0\% \\
		&       & 10    & 3568  & 3614  &0\% &0\% &0\% \\
		\cline{2-8}          & \multirow{3}[2]{*}{30} & 2     & 489   & 489   &0\% &0\% &0\% \\
		&       & 6     & 11435 & 11554 & 1.55\% & 4.16\% & 4.16\% \\
		&       & 10    & 3665  & 3700  & 3.72\% & 0.47\% & 0.47\% \\
		\cline{2-8}          & \multirow{3}[2]{*}{40} & 2     & 1754  & 1754  &0\% &0\% &0\% \\
		&       & 6     & 6228  & 6232  & 7.23\% & 4.19\% & 4.19\% \\
		&       & 10    & 7852  & 7891  & 9.87\% & 8.83\% & 8.83\% \\
		\cline{2-8}          & \multirow{3}[2]{*}{50} & 2     & 7446  & 7479  &0\% & 0.62\% & 0.62\% \\
		&       & 6     & 5696  & 5620  & 13.94\% & 2.78\% & 2.78\% \\
		&       & 10    & 4820  & 4777  & 14.38\% & 11.54\% & 11.54\% \\
		\hline
		\multirow{14}{*}{\texttt{Line}} & \multirow{2}[2]{*}{10} & 2     & 3     & 3     &0\% &0\% &0\% \\
		&       & 6     & 3     & 3     & 6.67\% & 6.67\% & 6.67\% \\
		\cline{2-8}          & \multirow{3}[2]{*}{20} & 2     & 8     & 8     &0\% &0\% &0\% \\
		&       & 6     & 317   & 317   &0\% &0\% &0\% \\
		&       & 10    & 38    & 38    &0\% &0\% &0\% \\
		\cline{2-8}          & \multirow{3}[2]{*}{30} & 2     & 489   & 489   &0\% &0\% &0\% \\
		&       & 6     & 12298 & 12395 & 3.20\% & 1.01\% & 1.01\% \\
		&       & 10    & 2773  & 2809  & 3.72\% & 0.23\% & 0.23\% \\
		\cline{2-8}          & \multirow{3}[2]{*}{40} & 2     & 1754  & 1754  &0\% &0\% &0\% \\
		&       & 6     & 7340  & 7363  & 7.86\% & 2.49\% & 2.49\% \\
		&       & 10    & 4916  & 4943  & 8.26\% & 0.26\% & 0.26\% \\
		\cline{2-8}          & \multirow{3}[2]{*}{50} & 2     & 7446  & 7504  &0\% & 0.62\% & 0.62\% \\
		&       & 6     & 5499  & 5395  & 16.97\% & 2.45\% & 2.45\% \\
		&       & 10    & 3063  & 3063  & 10.52\% &0\% &0\% \\
		\hline

\end{tabular}%
}
\caption{Averaged number of cuts in the incomplete formulations and MIP gaps for graphs \texttt{Comp}, \texttt{Cycle} and \texttt{Line}.\label{t:cuts1}}
\end{table}%
\begin{table}[H]
	\centering
	\adjustbox{max totalheight=0.9\textheight}{
		\begin{tabular}{|c|cr|rr|rrr|}
			\hline
		\texttt{Graph} & $n$     & $p$ & \texttt{\#Cuts\_\eqref{IC1}} &\texttt{\#Cuts\_\eqref{IC2}} & \texttt{gap\_NL} & \texttt{gap\_\eqref{IC1}} & \texttt{gap\_\eqref{IC2}}\\
			\hline
		\multirow{14}{*}{\texttt{Matching}} & \multirow{2}[2]{*}{10} & 2     & 3     & 3     &0\% &0\% &0\% \\
		&       & 6     & 0     & 0     &0\% &0\% &0\% \\
		\cline{2-8}          & \multirow{3}[2]{*}{20} & 2     & 8     & 8     &0\% &0\% &0\% \\
		&       & 6     & 7     & 7     &0\% &0\% &0\% \\
		&       & 10    & 1     & 1     &0\% &0\% &0\% \\
		\cline{2-8}          & \multirow{3}[2]{*}{30} & 2     & 489   & 489   &0\% &0\% &0\% \\
		&       & 6     & 1655  & 1678  & 5.46\% & 0.12\% & 0.12\% \\
		&       & 10    & 28    & 28    & 4.60\% &0\% &0\% \\
		\cline{2-8}          & \multirow{3}[2]{*}{40} & 2     & 1754  & 1754  &0\% &0\% &0\% \\
		&       & 6     & 4027  & 4055  & 16.31\% & 0.68\% & 0.68\% \\
		&       & 10    & 1673  & 1674  & 6.88\% & 0.10\% & 0.10\% \\
		\cline{2-8}          & \multirow{3}[2]{*}{50} & 2     & 7497  & 7544  &0\% & 0.62\% & 0.62\% \\
		&       & 6     & 1940  & 1940  & 23.84\% &0\% &0\% \\
		&       & 10    & 875   & 875   & 6.88\% &0\% &0\% \\
		\hline
		\multirow{14}{*}{\texttt{Star}} & \multirow{2}[2]{*}{10} & 2     & 3     & 3     &0\% &0\% &0\% \\
		&       & 6     & 2515  & 2507  & 13.33\% & 13.33\% & 13.33\% \\
		\cline{2-8}          & \multirow{3}[2]{*}{20} & 2     & 8     & 8     &0\% &0\% &0\% \\
		&       & 6     & 4456  & 4519  &0\% &0\% &0\% \\
		&       & 10    & 4369  & 4438  & 0.56\% &0\% &0\% \\
		\cline{2-8}          & \multirow{3}[2]{*}{30} & 2     & 489   & 489   &0\% &0\% &0\% \\
		&       & 6     & 1978  & 1994  & 6.33\% & 0.13\% & 0.13\% \\
		&       & 10    & 11739 & 11823 & 4.58\% & 1.96\% & 1.96\% \\
		\cline{2-8}          & \multirow{3}[2]{*}{40} & 2     & 1754  & 1754  &0\% &0\% &0\% \\
		&       & 6     & 11847 & 11941 & 9.32\% & 0.95\% & 0.95\% \\
		&       & 10    & 12071 & 12179 & 10.20\% & 1.42\% & 1.42\% \\
		\cline{2-8}          & \multirow{3}[2]{*}{50} & 2     & 7511  & 7522  &0\% & 0.62\% & 0.62\% \\
		&       & 6     & 4934  & 4865  & 13.64\% & 1.72\% & 1.72\% \\
		&       & 10    & 14225 & 14125 & 14.40\% & 1.23\% & 1.23\% \\
		\hline
		\multirow{14}{*}{\texttt{Star-Ring}} & \multirow{2}[2]{*}{10} & 2     & 3     & 3     &0\% &0\% &0\% \\
		&       & 6     & 100   & 100   & 23.33\% & 23.33\% & 23.33\% \\
		\cline{2-8}          & \multirow{3}[2]{*}{20} & 2     & 8     & 8     &0\% &0\% &0\% \\
		&       & 6     & 2951  & 2959  &0\% &0\% &0\% \\
		&       & 10    & 4167  & 4217  & 6.67\% & 6.67\% & 6.67\% \\
		\cline{2-8}          & \multirow{3}[2]{*}{30} & 2     & 491   & 491   &0\% &0\% &0\% \\
		&       & 6     & 11456 & 11535 & 1.68\% & 2.03\% & 2.03\% \\
		&       & 10    & 11121  & 11295 & 3.49\% & 1.28\% & 1.28\% \\
		\cline{2-8}          & \multirow{3}[2]{*}{40} & 2     & 1721  & 1721  &0\% &0\% &0\% \\
		&       & 6     & 7187  & 7230  & 3.44\% & 2.90\% & 2.90\% \\
		&       & 10    & 12675 & 12778 & 8.97\% & 1.36\% & 1.36\% \\
		\cline{2-8}          & \multirow{3}[2]{*}{50} & 2     & 7213  & 7326  &0\% & 0.62\% & 0.62\% \\
		&       & 6     & 6454  & 6381  & 10.49\% & 5.80\% & 5.80\% \\
		&       & 10    & 6210  & 6168  & 15.64\% & 1.72\% & 1.72\% \\
		\hline
	\end{tabular}%
}
\caption{Averaged number of cuts in the incomplete formulations and MIP gaps for graphs \texttt{Matching}, \texttt{Star} and \texttt{Star-Ring}.\label{t:cuts2}}
\end{table}%

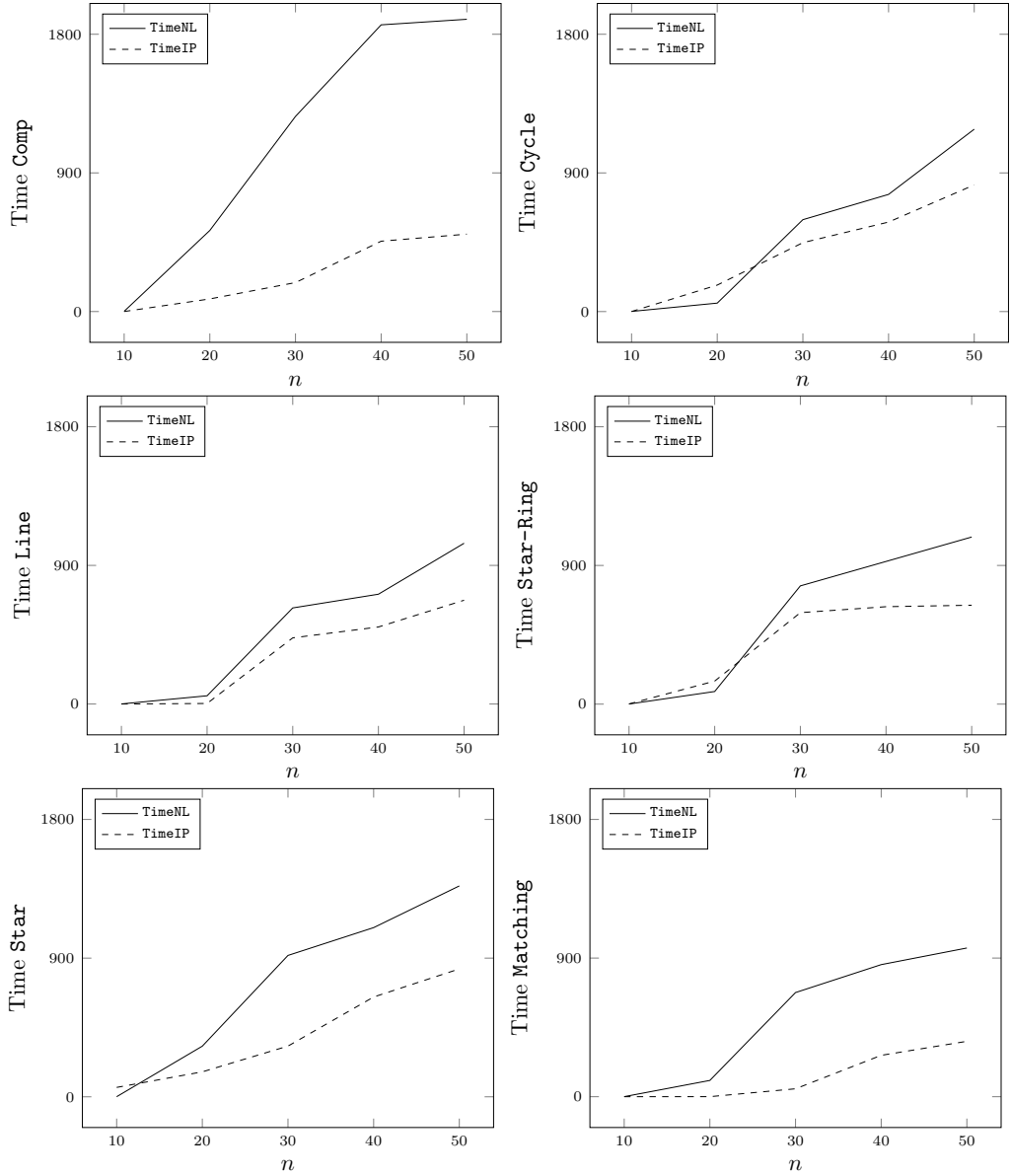
\begin{figure}[H]
   \adjustbox{max totalheight=0.8\textheight}{\begin{tikzpicture}[scale=0.8]
    \begin{axis}[
    xlabel={$n$},
    ymax=2000,
    ylabel={Time   \texttt{Comp}},
	xtick={0,10, 20, 30, 40, 50},
	xticklabels={0,10, 20, 30, 40, 50},
	ytick={0, 900, 1800, 2700, 3600},
	yticklabels={0, 900, 1800, 2700, 3600},
    legend pos=north west,
    xticklabel style={font=\tiny},
yticklabel style={font=\tiny}]
    ]
    \addplot [no marks] table [x=n, y=NL] {
n   NL    
0	0
10	1.378718233
20	526.4110812
30	1266.593341
40	1861.128207
50	1897.256109
};
    \addplot [no marks, dashed] table [x=n, y=NL] {    
n   NL    
0	0
10	0.085747167
20	81.87395378
30	187.9458588
40	456.5518347
50	502.4484196
};

\addlegendentry{\tiny\texttt{TimeNL}}
\addlegendentry{\tiny\texttt{TimeIP}}
    \end{axis}
    \end{tikzpicture}~\begin{tikzpicture}[scale=0.8]
    \begin{axis}[
    xlabel={$n$},
    ymax=2000,
    ylabel={Time   \texttt{Cycle}},
	xtick={0,10, 20, 30, 40, 50},
	xticklabels={0,10, 20, 30, 40, 50},
	ytick={0, 900, 1800, 2700, 3600},
	yticklabels={0, 900, 1800, 2700, 3600},
    legend pos=north west,
    xticklabel style={font=\tiny},
yticklabel style={font=\tiny}]
    ]
   \addplot [no marks] table [x=n, y=NL] {
n   NL    
0	0
10	0.9333426
20	54.85037594
30	596.5151607
40	760.5267706
50	1183.626357
};
    \addplot [no marks, dashed] table [x=n, y=NL] {
n   NL    
0	0
10	0.311870483
20	171.9403355
30	447.3044077
40	580.9036391
50	822.1401502
};

\addlegendentry{\tiny\texttt{TimeNL}}
\addlegendentry{\tiny\texttt{TimeIP}}
    \end{axis}
    \end{tikzpicture}}\\
    
    \adjustbox{max totalheight=0.8\textheight}{\begin{tikzpicture}[scale=0.8]
    \begin{axis}[
    xlabel={$n$},
    ymax=2000,
    ylabel={Time   \texttt{Line}},
	xtick={0,10, 20, 30, 40, 50},
	xticklabels={0,10, 20, 30, 40, 50},
	ytick={0, 900, 1800, 2700, 3600},
	yticklabels={0, 900, 1800, 2700, 3600},
    legend pos=north west,
    xticklabel style={font=\tiny},
yticklabel style={font=\tiny}]
    ]
   \addplot [no marks] table [x=n, y=NL] {
n   NL    
0	0
10	0.5956434
20	53.5401103
30	623.8237767
40	712.9073212
50	1043.717478
};
    \addplot [no marks, dashed] table [x=n, y=NL] {
n   NL    
0	0
10	0.089222617
20	3.615582389
30	429.5353879
40	500.1077315
50	673.7783997
};

\addlegendentry{\tiny\texttt{TimeNL}}
\addlegendentry{\tiny\texttt{TimeIP}}
    \end{axis}
    \end{tikzpicture}~ \begin{tikzpicture}[scale=0.8]
    \begin{axis}[
    xlabel={$n$},
    ymax=2000,
    ylabel={Time   \texttt{Star-Ring}},
	xtick={0,10, 20, 30, 40, 50},
	xticklabels={0,10, 20, 30, 40, 50},
	ytick={0, 900, 1800, 2700, 3600},
	yticklabels={0, 900, 1800, 2700, 3600},
    legend pos=north west,
    xticklabel style={font=\tiny},
yticklabel style={font=\tiny}]
    ]
   \addplot [no marks] table [x=n, y=NL] {
n   NL    
0	0
10	0.718876217
20	81.95137677
30	767.35
40	914.949789
50	1084.439362
};
    \addplot [no marks, dashed] table [x=n, y=NL] {
n   NL    
0	0    
10	1.230248133
20	149.4697511
30	593.6095651
40	632.2438837
50	641.8417396
};

\addlegendentry{\tiny\texttt{TimeNL}}
\addlegendentry{\tiny\texttt{TimeIP}}
    \end{axis}
    \end{tikzpicture}}
    
\adjustbox{max totalheight=0.8\textheight}{\begin{tikzpicture}[scale=0.8]
    \begin{axis}[
    xlabel={$n$},
    ymax=2000,
    ylabel={Time   \texttt{Star}},
	xtick={0,10, 20, 30, 40, 50},
	xticklabels={0,10, 20, 30, 40, 50},
	ytick={0, 900, 1800, 2700, 3600},
	yticklabels={0, 900, 1800, 2700, 3600},
    legend pos=north west,
    xticklabel style={font=\tiny},
yticklabel style={font=\tiny}]
    ]
   \addplot [no marks] table [x=n, y=NL] {
n   NL    
0	0
10	0.708905583
20	328.7619082
30	917.5593277
40	1097.987947
50	1368.537804
};
    \addplot [no marks, dashed] table [x=n, y=NL] {
n   NL    
0	0    
10	61.0318166
20	161.9971762
30	329.1479002
40	647.8329025
50	829.9491042
};

\addlegendentry{\tiny\texttt{TimeNL}}
\addlegendentry{\tiny\texttt{TimeIP}}
    \end{axis}
    \end{tikzpicture}~\begin{tikzpicture}[scale=0.8]
    \begin{axis}[
    xlabel={$n$},
    ymax=2000,
    ylabel={Time   \texttt{Matching}},
	xtick={0,10, 20, 30, 40, 50},
	xticklabels={0,10, 20, 30, 40, 50},
	ytick={0, 900, 1800, 2700, 3600},
	yticklabels={0, 900, 1800, 2700, 3600},
    legend pos=north west,
    xticklabel style={font=\tiny},
yticklabel style={font=\tiny}]
    ]
   \addplot [no marks] table [x=n, y=NL] {
n   NL    
0	0
10	0.454350383
20	105.6764069
30	676.2012162
40	857.4004799
50	965.3544121
};
    \addplot [no marks, dashed] table [x=n, y=NL] {
n   NL    
0	0
10	0.0526796
20	0.338336744
30	51.45913552
40	268.6431282
50	359.6096521
};

\addlegendentry{\tiny\texttt{TimeNL}}
\addlegendentry{\tiny\texttt{TimeIP}}
    \end{axis}
    \end{tikzpicture} } 
    
\caption{Graphics of averaged CPU times for the different graphs by number of demand points.\label{fig:times}}
\end{figure}

\section{Conclusions}\label{sec:6}

In this paper, we analyzed a novel version of the Continuous Maximal Covering Location Problem in which the facilities are required to be linked through a given graph structure and such that its distance does not exceed a given distance limit. We provide a general framework for the problem for any finite dimensional space and any norm-based distance. We provide a MISOCO formulation for the problem. We further analyze the geometry of the problem and proved that the continuous variables of the formulation can be projected out and the nonlinear constraints can be replaced by polynomially many linear constraints, resulting in a compact ILP model for the problem. The geometry of planar instances is exploited for a suitable construction of the formulation. After testing the performance of this ILP formulation, and determining that the computational load of adding its linear constraints, we derive two branch-\&-cut approaches to efficiently solve the problem. The performance of the procedures is tested with an extensive battery of computational experiments showing that the branch-and-cut approaches outperforms the rest of the approaches. 

Further research on the topic includes, among others, the design of heuristic approaches for solving the problem or the extension of the \textit{interconnection} framework to other continuous multifacility location problems, as ordered median location problems \cite{BPE16,NP06}. A deeper analysis of the $\O$-sets introduced in this paper as well as finding valid inequalities for \eqref{IP} will shed light into new geometrical approaches for solving MCLPIF for larger instances. The use of block norms (\cite{WW85}) simplifies $\O$-sets to polytopes, which may ease the analysis.

\section*{Acknowledgements}

The authors were partially supported by research group SEJ-584 (Junta de Andaluc\'ia). The first author was also supported by Spanish Ministry of Education and Science/FEDER grant number MTM2016-74983-C02-(01-02), and projects FEDER-US-1256951, CEI-3-FQM331 and \textit{NetmeetData}: Ayudas Fundaci\'on BBVA a equipos de investigaci\'on cient\'ifica 2019. The second author was supported by Spanish Ministry of Education and Science grant number PEJ2018-002962-A.

\end{document}